\documentclass[11pt]{article}
\usepackage{mathrsfs}
\usepackage{mathrsfs}
\usepackage{mathrsfs}
\usepackage{amsfonts}
\usepackage{amssymb}
\usepackage{amsmath}
\usepackage{amsthm}
\usepackage{color}
\theoremstyle{plain}
\newtheorem{thm}{Theorem}[section]

\newtheorem{lem}[thm]{Lemma}
\newtheorem{prop}[thm]{Proposition}

\newtheorem{Problem}{Problem}[section]

\theoremstyle{definition}
\newtheorem{defn}{Definition}[section]
\newtheorem{exam}{Example}[section]
\theoremstyle{remark}
\newtheorem{rem}{Remark}[section]

\begin{document}
\title{{\Large\bf {Arveson's version of the Gauss-Bonnet-Chern formula for Hilbert modules over the polynomial rings\uppercase\expandafter{}}}
\thanks{
This work is supported by National Natural Science Foundation of China : (12271298 and 11871308).}}
\author{{\normalsize Penghui Wang, Ruoyu Zhang and Zeyou Zhu} \\
{ School of Mathematics, Shandong University,} {\normalsize Jinan, 250100, China} }
\date{}
 \maketitle
\begin{abstract}
In this paper, we complete the framework of Arveson's version of the Gauss-Bonnet-Chern formula by proving that Arveson's version of the Gauss-Bonnet-Chern formula holding true for the quotient module in the Drury-Arveson module  is equivalent to the associated submodule being locally algebraic. Moreover, we establish the asymptotic Arveson's curvature invariant and the asymptotic Euler characteristic for contractive Hilbert modules over the polynomial ring in infinitely many variables, and obtain the infinitely-many-variables analogue of Arveson's version of Gauss-Bonnet-Chern formula. Finally, we solve the finite defect problem for submodules of the Drury-Arveson module $H^2$
  in infinitely many variables by proving that $H^2$ has no nontrivial submodules of finite rank.
\end{abstract}

\bigskip


\bigskip

\noindent{\bf Key Words:} Drury-Arveson module, Arveson's version of the Gauss-Bonnet-Chern formula, locally algebraic submodules, finite defect problem

 \maketitle
\numberwithin{equation}{section}
\newtheorem{theorem}{Theorem}[section]
\newtheorem{lemma}[theorem]{Lemma}
\newtheorem{proposition}[theorem]{Proposition}
\newtheorem{corollary}[theorem]{Corollary}
\newtheorem{claim}{Claim}
\newtheorem{Case}{Case}
\newtheorem*{Notation}{Notation}


\section{Introduction}

Let $\mathbb B_d$ be the open unit ball in $\mathbb C^d$. Recall the Drury-Arveson space $H_d^2$ is a reproducing kernel Hilbert space of holomorphic functions on $\mathbb B_d$, with the reproducing kernel
$$
\mathcal{K}_{\lambda}(z)={1\over 1-\langle z,\lambda\rangle}.
$$ It can be viewed as a Hilbert module \cite{CG,DP} on the polynomial ring $\mathbb C[z_1,\cdots,z_d]$, equipped with a natural module action defined by the multiplication by polynomials.
To simplify the notations, write $\mathcal{P}_d=\mathbb{C}[z_1,\cdots,z_d].$ In \cite{Arveson curvature}, Arveson introduced the curvature invariant for quotient modules in $H_d^2\otimes \mathbb C^{N} (N\in \mathbb N^+)$. Let $\mathcal{M}\subset H_d^2\otimes \mathbb C^{N} $ be a submodule and $\mathcal{M}^\perp$ be the corresponding quotient module. Let $S_{z_i}$ be the compression of $M_{z_i}\otimes I_N$ onto the quotient space $\mathcal{M}^\perp$, i.e.
$$
S_{z_i} f=P_{\mathcal{M}^\perp}(M_{z_i}\otimes I_N) f, \, \forall f\in \mathcal{M}^\perp.
$$
 For $z\in\mathbb B_d$, define $S_d(z)=\sum\limits_{i=1}^{d}\bar{z_i}S_{z_i}$, and
$$
F(z)=\Delta_{\mathcal{M}^\perp} (I-S_d(z)^*)^{-1}(I-S_d(z))^{-1}\Delta_{\mathcal{M}^\perp},
$$
where $\Delta_{\mathcal{M}^\perp}=(I-S_{z_1}S_{z_1}^*-\cdots-S_{z_d}S_{z_d}^*)^{1\over 2}$ is the defect operator for $\mathcal{M}^\perp.$ Then for all $z\in\mathbb B_d$, $F(z)$ are positive operators on the range of $\Delta_{\mathcal M^\perp}$. In what follows,  the range of a bounded operator $A$ on a Hilbert space $\cal H$ will be denoted by $\operatorname{ran} A$. It was shown in \cite{Arveson curvature} that the radial limits $\lim\limits_{r\to 1^-}(1-r^2) \operatorname{trace}(F(r\xi))$ exist almost everywhere on the boundary $\partial \mathbb B_d$ relative to $d\sigma_d$ on $\partial \mathbb{B}_{d}$, and Arveson's curvature invariant is defined by
$$
K(\mathcal{M}^\perp)=\int_{\partial \mathbb B_d}\lim\limits_{r\to 1^-}(1-r^2) \operatorname{trace}(F(r\xi))d\sigma_d(\xi),
$$
where $\sigma_d$ is the rotationally invariant probability measure on $\partial \mathbb{B}_{d}$. Such an invariant was deeply studied by \cite{Arveson Dirac, Bhattacharyya, Bhattacharyya2, Fang Xiang, GRS, inner multipliers, Levy, Muhly, Parrott} and references therein.

Next, we will introduce the Euler characteristic. 
Let $$
\mathbb{M}_{\mathcal{M}^\perp}=\operatorname{span}\left\{f(S_{z_1},\cdots,S_{z_d}) \cdot \Delta_{\mathcal{M}^\perp} \zeta \mid f \in \mathcal{P}_d, \zeta \in \mathcal{M}^\perp\right\}.
$$
Then $\mathbb{M}_{\mathcal{M}^\perp}$ has finite free resolutions in the category of finitely generated $\mathcal{P}_d$-modules
 \begin{eqnarray}\label{Euler-ch}
 0\rightarrow F_n\rightarrow\cdots\rightarrow F_2\rightarrow F_1\rightarrow \mathbb{M}_{\mathcal{M}^\perp}\rightarrow 0,\end{eqnarray}
 each $F_k$ being a sum of $\beta_k$ copies of the module $\mathcal{P}_d$.
 In \cite{Arveson curvature}, Arveson defined the Euler characteristic of $\mathcal{M}^\perp$ by
\begin{equation}\label{chenlaoshi}
 \chi(\mathcal{M}^\perp)=\sum\limits_{k=1}^{n}(-1)^{k+1}\beta_k.
\end{equation}

In \cite[Theorem B]{Arveson curvature}, Arveson proved for any graded quotient module $\mathcal{M}^\perp$ in $ H^2_d\otimes \mathbb C^N$,
\begin{eqnarray}\label{Koula}
K(\mathcal{M}^\perp)=\chi(\mathcal{M}^\perp).
\end{eqnarray}
From Arveson's viewpoint, such a formula can be seen as an operator-theoretic version of the Gauss-Bonnet-Chern formula of Riemannian geometry, and we call it Arveson's version of the Gauss-Bonnet-Chern formula.

In the beginning of \cite[Section 5]{Arveson curvature}, Arveson wanted to study such a formula in general case, which is reformulated as the following problem.
\begin{Problem}\label{pro1.1}
For which quotient modules of $H^2_d \otimes \mathbb{C}^N$
 does Arveson's version of the Gauss-Bonnet-Chern formula hold?
\end{Problem}
Let $\mathcal{M}$ be a submodule of $H_d^{2}\otimes \mathbb{C}^N,$
and
\begin{equation}\label{bianhao1}
    E_{\lambda} \mathcal{M}=\{f(\lambda):f\in \mathcal{M}\}.
    \end{equation}
The fiber dimension of $\mathcal{M},$ denoted by $fd(\mathcal{M}),$ is defined by
\begin{equation}\label{bianhao11}
   fd(\mathcal{M})=\operatorname{sup} \{\operatorname{dim}E_{\lambda}\mathcal{M}:\lambda\in \mathbb{B}_d\}.
    \end{equation}
    Set
    \begin{equation}\label{bianhao111}
   \text{mp}(\mathcal{M})=\{\lambda \in \mathbb{B}_d: \operatorname{dim}E_{\lambda}\mathcal{M}=fd(\mathcal{M})\},
    \end{equation}
which is called the set of maximal points for $\mathcal{M}$ in \cite{CF}.
There are many efforts to study the fiber dimension of submodules \cite{CL2, CF, CCF, CGW, EM, F2, F3, F4, GRS}, and here
 we introduce locally algebraic submodules in terms of $\operatorname{mp}(\mathcal{M})$.
\begin{defn}
A submodule $\mathcal{M}$ is said to be locally algebraic if there exists a $\lambda_{0}\in \operatorname{mp}(\mathcal{M})$ and polynomials $\{p_i\}_{i=1}^\mathfrak{m} \subseteq \mathcal{M}$ such that
  $$E_{\lambda_{0}}\mathcal{M}=\operatorname{span}\{p_i(\lambda_{0}):1\leq i\leq \mathfrak{m}\}.$$
  \end{defn}
  \begin{rem}\label{recall}
  Recall that $\mathcal{M}$ is algebraic if $\mathcal{M}$ is generated by polynomials. Obviously, algebraic submodules are locally algebraic.
  \end{rem}
\begin{thm}\label{thm1.1}
For a submodule $\mathcal{M}\subseteq H_d^2 \otimes \mathbb{C}^N$, Arveson's version of the Gauss-Bonnet-Chern formula holds true for $\mathcal{M}^\perp$ iff $\mathcal{M}$ is locally algebraic.
\end{thm}
In \cite{Arveson curvature}, Arveson proved that the quotient modules of {\color{red}$H_d^2 \otimes \mathbb{C}^N$} have universal properties. To illustrate it, let $(T_1,T_2,\cdots, T_d)$ be {\color{red}a tuple of commuting operators} on a Hilbert space $\mathcal{H},$ which is called a $d$-contraction
if $\sum\limits_{i=1}^{d}T_iT_i^*\leq I.$ In this case, the Hilbert space $\mathcal{H}$ can be viewed as a module over the polynomial ring $\mathcal{P}_d$ with the module action given by
$$f\cdot\zeta=f(T_1,  \cdots, T_d )\zeta, \quad f\in \mathcal{P}_d, \quad \zeta\in \mathcal{H}.$$
Obviously, Drury-Arveson's module $H_d^2$ is a $d$-contractive Hilbert module and   Arveson \cite{Arveson curvature} proved that any pure $d$-contractive Hilbert module is unitarily equivalent to a quotient module of  $H_d^2 \otimes \mathbb{C}^N$ for some $N>0$.
\begin{rem}
\begin{itemize}
\item[(1)]When $N=1,$ Theorem \ref{thm1.1} recovers Arveson's result \cite[Theorem E and Proposition 7.4]{Arveson curvature} that $K(M^\perp) = \chi(M^\perp)$
if and only if $\mathcal{M}$ contains a nonzero polynomial.
\item[(2)]
\cite[Theorem 18]{Fang Xiang} demonstrates that Arveson's version of the Gauss-Bonnet-Chern formula holds when the submodules are generated by polynomials. Furthermore, as indicated by Remark \ref{recall}, this can be regarded as a corollary of Theorem \ref{thm1.1}.
\item[(3)]
For $N\geq 2,$ by \cite[Theorem B]{Arveson curvature} and Theorem \ref{thm1.1}, graded submodules in $ H^2_d\otimes \mathbb C^N$ are locally algebraic, and hence they have enough polynomials. It is noteworthy that, to the best of our knowledge, it was previously unknown whether a graded submodule in Arveson's framework contains a nonzero polynomial.
\end{itemize}
\end{rem}

Our next goal is to extend Arveson's version of the Gauss-Bonnet-Chern formula to Drury-Arveson modules in infinitely many variables.
 To do this, we need to define the asymptotic curvature invariant and the asymptotic Euler characteristic for the so-called  finite rank $\omega$-contractive Hilbert modules over {\color{red}the polynomial ring} in infinitely many variables.
{\color{red}A sequence of commuting operators} $\mathbb T_\omega=(T_1,\cdots, T_m,\cdots)$  is called an {\it $\omega$-contraction}, if
\begin{eqnarray}
\sum\limits_{i=1}^m T_iT_i^*\leq I, ~ \forall~ m\geq1.
\end{eqnarray}
Let $ \mathcal{P}_{\infty}=\mathop{\cup}\limits_{m=1}^{\infty} \mathcal{P}_m$ be {\color{red}the polynomial ring in infinitely many variables}.
  Any $\omega$-contraction $\mathbb T_\omega$ acting on a Hilbert space $\mathcal{H}$ gives rise to a $\mathcal{P}_{\infty}$-module structure on $\mathcal{H}$
in the natural way, as follows:
\begin{eqnarray}
p\cdot\zeta=p(T_1,  \cdots, T_m, \cdots )\zeta, \quad p\in \mathcal{P}_{\infty}, \quad \zeta\in \mathcal{H}.
\end{eqnarray}
In this case, $\mathcal{H}$ is called an {\it $\omega$-contractive Hilbert module}.
Obviously, if $\mathbb T_\omega$ is an $\omega$-contraction, then $\sum\limits_{i=1}^m T_iT_i^*$ converges in the strong operator topology as $m$ tends to infinity. The defect operator $\Delta_\mathcal{H}$ is defined by
$$
\Delta_\mathcal{H} = \Big( I - \text{SOT-}\lim\limits_{k\rightarrow \infty} \sum\limits_{i=1}^{k} T_{i} T_{i}^{*} \Big)^{1 / 2},
$$
and
{\it$\operatorname{dim} \operatorname{ran}\Delta_\mathcal{H}$ is called the rank of $\mathcal{H}.$}
Moreover, the $\omega$-contraction $\mathbb T_{\omega}$ gives rise to a completely  positive map $\phi$ on $\mathcal{B}(\mathcal{H})$ by way of
$$
\phi(A)=\text{SOT-}\lim\limits_{k\rightarrow \infty} \sum\limits_{i=1}^{k} T_{i}A T_{i}^{*}, \quad A \in \mathcal{B}(\mathcal{H}) .
$$
Obviously, $\phi^n(I)$ is decreasing.  An $\omega$-contractive Hilbert module is pure if
$$\text{SOT-} \lim\limits _{n \rightarrow \infty} \phi^{n}(I)=0.$$
As an example of $\omega$-contractive Hilbert modules, Arveson's module of infinitely many variables will be reviewd in Section 3 and denoted by $H^2.$ Analogous to the finitely-many-variables case, it will be proved in Section 4, any pure $\omega$-contractive Hilbert module of finite rank is unitarily equivalent to a quotient module $\cal M^\perp$ of $H^2 \otimes \mathbb{C}^N$
  for some $N \geq 1.$
As a reproducing kernel Hilbert space of holomorphic functions on the unit ball $\mathbb{B}$ in $\mathbb \ell^2$, with the reproducing kernel
$$
{\cal K}_\lambda(z)={1\over 1-\langle z,\lambda\rangle}, \quad \lambda\in\mathbb B,
$$
$H^2$ was introduced by Agler and M$^{\text{c}}$Carthy \cite{Pick interpolation} to study interpolation problems.

As in the finitely-many-variables case, for any submodule $\cal M$ in $H^2\otimes \mathbb C^N$, the compression of $M_{z_i}\otimes I_N$ on $\cal M^\perp$ is defined by
$$
S_{z_i} f=P_{\mathcal{M}^\perp}(M_{z_i}\otimes I_N) f, \, \forall f\in \mathcal{M}^\perp,
$$
and
$$
S(z)=\text{SOT-}\lim\limits_{k\rightarrow \infty} \sum\limits_{i=1}^{k}\bar{z}_{i}S_{z_i} .
$$
 Moreover, for any $z\in \mathbb{B}$,
the operator $I-S(z)$ is invertible. Define
$$
\label{F}F(z) \xi=\Delta_{\mathcal{M}^\bot}\left(I-S(z)^{*}\right)^{-1}(I-S(z))^{-1} \Delta_{\mathcal{M}^\bot} \xi, \quad \xi \in \text{ran}\,\Delta_{\mathcal{M}^\bot}.
$$
For any integer $m>0$, let
$$
K_{0}^m(z^{(m)})=\lim _{r \rightarrow 1}\left(1-r^{2}\right) \operatorname{trace} F(r z_1,\cdots,rz_m,0,0,\cdots)
$$
for almost every point $z^{(m)}=(z_1,\cdots,z_m) \in \partial \mathbb{B}_{m}$ relative to $d \sigma_m$ on $\partial \mathbb{B}_{m}.$ Set $$K_m({\mathcal{M}^\bot})=\int_{\partial \mathbb{B}_{m}} K_{0}^m(z^{(m)}) d \sigma_m.$$
In Section 4,
we will prove that $$K_{m+1}({\mathcal{M}^\bot})\leq K_{m}({\mathcal{M}^\bot})\leq\text{rank}({\mathcal{M}^\bot}).$$
Hence for ${\mathcal{M}^\bot},$ we define the asymptotic curvature invariant
by $$K({\mathcal{M}^\bot})=\lim \limits_{m \rightarrow \infty} K_{m}({\mathcal{M}^\bot}).$$
Seemingly, the asymptotic curvature invariant depends on the choice of the orthonormal basis of $\ell^2.$ However the following result will imply that the asymptotic curvature invariant is coordinate-free, which
is an infinitely-many-variables analogue of \cite[Theorem 5.2]{inner multipliers}.

\begin{thm}\label{thm1.3}
Let $\mathcal{M}\subseteq H^2\otimes \mathbb {C}^N$ be a submodule, then
$$K(\mathcal{M}^\bot)=N-fd(\mathcal{M}).$$
\end{thm}

Here, $ E_{\lambda} \mathcal{M}$ and $fd(\mathcal{M})$ can be defined similarly to
\eqref{bianhao1} and \eqref{bianhao11}. Next, we define the asymptotic Euler characteristic. For an $\omega$-contractive
Hilbert module $\mathcal{H}$ of finite rank,
 let
$$
\mathbb{M}_\mathcal{H}^m=\operatorname{span}\left\{f \cdot \Delta_\mathcal{H} \zeta : f \in \mathcal{P}_m, \zeta \in \mathcal{H}\right\}.
$$
Then $\mathbb{M}_\mathcal{H}^m$ is a finitely generated $\mathcal{P}_m$-module, which
has finite free resolutions.
   And hence
  $\chi_m(\mathcal{H})$ can be defined similarly in \eqref{chenlaoshi}.
  It will be proved in Section 4 that $\chi_m(\mathcal{H})$ is decreasing on $m$, and the asymptotic Euler characteristic will be defined as
 \begin{eqnarray}
 \chi(\mathcal{H})=\lim\limits_{m\to\infty} \chi_m(\mathcal{H}).
 \end{eqnarray}
The following result solves Problem \ref{thm1.1} in the infinitely-many-variables case.
\begin{thm}\label{thm1.5}
For a submodule $\mathcal{M}\subseteq H^2 \otimes \mathbb{C}^N$, Arveson's version of the Gauss-Bonnet-Chern formula holds true for $\mathcal{M}^\perp$ iff $\mathcal{M}$ is locally algebraic.
\end{thm}
 It is well known that any nontrivial homogeneous submodule $\mathcal{M}$ of $H_d^2 $
  must be generated by polynomials. Consequently, one has $K(\mathcal{M}^\bot) = \chi(\mathcal{M}^\bot).$
 However, in the infinitely-many-variables case, we present an example of a homogeneous submodule in $H^2$ that does not contain any nonzero polynomial.
By Theorem \ref{thm1.5}, Arveson's version of the Gauss-Bonnet-Chern formula
may not hold for homogeneous quotient modules in $H^2$.


At last, we will study $\it{the \,finite\, defect\, problem}$ for submodules in $H^2.$ {\color{red}Such a problem for submodules in $H_d^2$
 was raised by Arveson in \cite[Page 226]{Arveson curvature} and also makes sense for other reproducing kernel Hilbert spaces, such as Hardy space and Bergman space over the polydisk and the unit ball, as studied by Guo \cite{Gu2}.}
\begin{Problem}\label{problem2}
   Is the rank of each nonzero submodule of $H_d^2$ that has infinite codimension in $H_d^2$
  infinite?
\end{Problem}
By showing that for $d \geq 2$, any nonzero submodule $\mathcal{M} \subseteq H_d^2$
  of finite rank has finite codimension in $H_d^2,$ Guo \cite{Gu} provided a positive answer to Problem \ref{problem2}. In the infinitely-many-variables case,
we have the following theorem.
\begin{thm}\label{fulelao}
If $\mathcal{M}$ is a nonzero submodule of $H^2$ of finite rank, then $\mathcal{M}=H^2.$
\end{thm}

This paper is organized as follows.
In Section 2, we prove Theorem \ref{thm1.1}, where a key step is a new algebraic representation of the Euler characteristic.
In Section 3, fundamental properties of Drury-Arveson's space of infinitely many variables are introduced.
In Section 4, we define the asymptotic curvature invariant and the asymptotic Euler characteristic for $\omega$-contractive Hilbert modules of finite rank, and then prove Theorem \ref{thm1.5}.
In Section 5, we present an example of a homogeneous submodule in $H^2$
  that contains no nonzero polynomials, demonstrating that Arveson's version of the Gauss-Bonnet-Chern formula may not necessarily hold for homogeneous quotient modules in $H^2$.
  In Section 6, we prove Theorem \ref{fulelao}.

\noindent

\section{Arveson's version of the Gauss-Bonnet-Chern formula}
This section is dedicated to proving our main result, Theorem \ref{thm1.1}.
The main tool used to deal with Theorem \ref{thm1.1} is the representation of the Euler characteristic in terms of the rank of finitely generated modules over an integral domain.
Let $\mathfrak{R}$ be an integral domain, and $S = \mathfrak{R} \setminus \{0\}.$ Then $S^{-1}\mathfrak{R}$
  is the field of fractions of $\mathfrak{R}$. For a finitely generated module $F$ over $\mathfrak{R}$, let $S^{-1}F$
  denote the module of fractions. It is evident that $S^{-1}F$
  forms a vector space over $S^{-1} \mathfrak{R}$. We denote the rank of $F$ by $\dim_{S^{-1} \mathfrak{R}}(S^{-1}F)$, as introduced in \cite[Page 873]{Rotman}. Obviously, in \eqref{Euler-ch}, $\beta_k$ is the rank of $F_k.$
For a $d$-contractive Hilbert module $\mathcal{H},$ the defect operator is defined by
$$
\Delta_\mathcal{H}=\left(I-T_{1} T_{1}^{*}-\cdots-T_{d} T_{d}^{*}\right)^{1 / 2}.
$$
The $d$-contractive Hilbert module $\mathcal{H}$ is of finite rank if $\Delta_\mathcal{H}$ is a finite rank operator. Similar to
\eqref{Euler-ch}, the Euler characteristic can also be defined.
Recall that in \cite{Arveson curvature} Arveson defined
\begin{eqnarray}\label{M_{H,d}}
\mathbb{M}_{\mathcal{H}}=\operatorname{span}\left\{f \cdot \Delta_\mathcal{H} \zeta \mid f \in \mathcal{P}_d, \zeta \in \mathcal{H}\right\}.
\end{eqnarray}
By \cite[Lemma 10.1]{Rotman} and \cite[Proposition 2.11]{Atiyah}, it can be clearly seen that
\begin{equation}\label{oulashu}
  \chi(\mathcal{H})=\operatorname{rank}(\mathbb{M}_{\mathcal{H}}),
\end{equation}
 which plays the key role in the study of Arveson's version of the Gauss-Bonnet-Chern formula.
Before proving Theorem \ref{thm1.1},
 we need some lemmas.

\begin{lem}\label{dimdefect}
Let $\mathcal{M}$ be a
submodule in $H_d^2 \otimes \mathbb{C}^N$ such that $$\operatorname{dim}~\mathcal{M}\cap (\mathbb{C} \otimes \mathbb{C}^N)=n.$$ Then for an orthonormal basis $\{1 \otimes e_i\}_{i=n+1}^{N}$ of
$\mathbb{C} \otimes \mathbb{C}^N \ominus (\mathcal{M}\cap (\mathbb{C} \otimes \mathbb{C}^N)),$
 $\{P_{\mathcal{M}^\perp}(1 \otimes e_i)\}_{i=n+1}^{N}$ is a linear basis of $\operatorname{ran}\,\Delta_{\mathcal{M}^\perp},$ in particular, $\operatorname{rank}\,\Delta_{\mathcal{M}^\perp}=N-n.$
\end{lem}
\begin{proof}
At first,
$$
\begin{aligned}
\Delta_{\mathcal{M}^\perp}^2&=I_{\mathcal{M}^\perp}-\sum\limits_{i=1}^{d}S_{z_i}S_{z_i}^*\\&=
I_{\mathcal{M}^\perp}-\sum\limits_{i=1}^{d}P_{\mathcal{M}^\bot} M_{z_i}P_{\mathcal{M}^\bot} M_{z_i}^*P_{M^\bot}|_{M^\bot}
\\&=
P_{\mathcal{M}^\bot}\left(I-\sum\limits_{i=1}^{d} M_{z_i}M_{z_i}^*\right)\bigg|_{\mathcal{M}^\bot}
\\&=P_{\mathcal{M}^\bot}(E_0 \otimes I)P_{\mathcal{M}^\bot}|_{\mathcal{M}^\bot},
\end{aligned}
$$
where $E_{0}$ is the orthogonal projection from $H_d^{2}$ onto the one-dimensional space of constant functions.
Now, take an orthonormal basis $\{1 \otimes e_i\}_{i=1}^{n}$ of $ \mathcal{M}\cap (\mathbb{C}\otimes \mathbb{C}^N),$ which extends to an orthonormal basis $\{1 \otimes e_i\}_{i=1}^{N}$ of $\mathbb{C}\otimes \mathbb{C}^N.$

Next, we claim that both $\{P_{\mathcal{M}^\bot}(1\otimes e_i)\}_{i=n+1}^{N}$ and $\{P_{\mathcal{M}^\bot}(E_0 \otimes I)P_{\mathcal{M}^\bot}(1\otimes e_i)\}_{i=n+1}^{N}$
are linearly independent. Firstly, assume that for $\mu_i \in \mathbb{C}$ such that
$$\sum\limits_{i=n+1}^{N}\mu_i P_{\mathcal{M}^\bot} (1 \otimes e_i)=0,$$ and we have $\sum\limits_{i=n+1}^{N}\mu_i (1 \otimes e_i) \in \mathcal{M} \cap (\mathbb{C}\otimes \mathbb{C}^N).$
By the choice of $\{1 \otimes e_i\}_{i=n+1}^{N},$ $$\sum\limits_{i=n+1}^{N}\mu_i (1 \otimes e_i) \perp \left(\mathcal{M} \cap (\mathbb{C}\otimes \mathbb{C}^N)\right).$$
It follows that $\sum\limits_{i=n+1}^{N}\mu_i (1 \otimes e_i)=0,$ and hence $\mu_i=0.$
Similarly, assume that $\mu_i \in \mathbb{C}$ such that $$\sum\limits_{i=n+1}^{N}\mu_i P_{\mathcal{M}^\bot}(E_0 \otimes I)P_{\mathcal{M}^\bot}(1\otimes e_i)=0.$$
Then $$\sum\limits_{i=n+1}^{N}\mu_i \left(E_0 \otimes I\right) P_{ \mathcal{M}^{\perp}} (1\otimes e_i)\in \mathcal{M}.$$ Obviously, $\operatorname{ran}\,(E_0 \otimes I)\subseteq \mathbb{C}\otimes \mathbb{C}^N,$
and hence $$ \left(E_0 \otimes I\right) P_{ \mathcal{M}^{\perp}}  \sum\limits_{i=n+1}^{N}\mu_i(1\otimes e_i)\in \mathcal{M}\cap (\mathbb{C}\otimes \mathbb{C}^N).$$
Therefore, by the choice of $\{1 \otimes e_i\}_{i=n+1}^{N}$ again, 
$$
\begin{aligned}
0&= \left\langle\left(E_0 \otimes I\right) P_{\mathcal{M}^{\perp}}  \sum\limits_{i=n+1}^{N}\mu_i(1\otimes e_i), \sum\limits_{i=n+1}^{N}\mu_i(1\otimes e_i) \right\rangle\\&=
\left\langle P_{ \mathcal{M}^{\perp}}  \sum\limits_{i=n+1}^{N}\mu_i(1\otimes e_i), \sum\limits_{i=n+1}^{N}\mu_i(1\otimes e_i) \right\rangle
\\&=
\left\langle P_{ \mathcal{M}^{\perp}}  \sum\limits_{i=n+1}^{N}\mu_i(1\otimes e_i),  P_{ \mathcal{M}^{\perp}} \sum\limits_{i=n+1}^{N}\mu_i(1\otimes e_i) \right\rangle
\\&=
\left\|P_{ \mathcal{M}^{\perp}}  \sum\limits_{i=n+1}^{N}\mu_i(1\otimes e_i)\right\|^2,
\end{aligned}
$$
and hence $\mu_i=0.$ The claim is proved. Then
$$
\begin{aligned}
\operatorname{dim} &\operatorname{span} \,\{P_{\mathcal{M}^\bot}(E_0 \otimes I)P_{\mathcal{M}^\bot}(1\otimes e_i)\}_{i=n+1}^{N}\\&=\operatorname{dim}\operatorname{span}\,\{P_{\mathcal{M}^\bot}(1\otimes e_{n+1}),\cdots, P_{\mathcal{M}^\bot}(1\otimes e_{N})\}\\&=N-n.
\end{aligned}
$$
Obviously,
$$
\begin{aligned}
\operatorname{span} \,\{P_{\mathcal{M}^\bot}(E_0 \otimes I)P_{\mathcal{M}^\bot}(1\otimes e_i)\}_{i=n+1}^{N}&\subseteq \{P_{ \mathcal{M}^{\perp}} \left(E_0 \otimes I\right) f:f \in \mathcal{M}^\perp \}\\&
\subseteq \operatorname{span}\,\{P_{\mathcal{M}^\bot}(1\otimes e_{n+1}),\cdots, P_{\mathcal{M}^\bot}(1\otimes e_{N})\}.
\end{aligned}
$$
Therefore
$$
\begin{aligned}
\operatorname{ran}\,\Delta_{\mathcal{M}^\perp}&=\operatorname{ran}\,\Delta_{\mathcal{M}^\perp}^2\\&=\{P_{ \mathcal{M}^{\perp}} \left(E_0 \otimes I\right) f:f \in \mathcal{M}^\perp \}
\\&=\operatorname{span}\,\{P_{\mathcal{M}^\bot}(1\otimes e_{n+1}),\cdots, P_{\mathcal{M}^\bot}(1\otimes e_{N})\}.
\end{aligned}
$$
The proof is completed.
%
%
%
\end{proof}
In particular,
\begin{equation}\label{(2.1)}
   \mathbb{M}_{\mathcal{M}^\perp}
=\operatorname{span}\left\{f \cdot P_{\mathcal{M}^\perp}(1 \otimes e_i) : f \in  \mathcal{P}_d,  n+1\leq i \leq N\right\}.
\end{equation}

\begin{lem}\label{KNf}
Let $\mathcal{M}$ be a
submodule of $H_d^2 \otimes \mathbb{C}^N.$ Then
$$K(\mathcal{M}^\bot)=N -fd(\mathcal{M}).$$
\end{lem}
\begin{proof}
By \cite[Theorem 5.2]{inner multipliers},
\begin{eqnarray}\label{eq:K}
K(\mathcal{M}^\perp)=\operatorname{dim}\operatorname{ran}\,\Delta_{\mathcal{M}^\perp}-\sup \{\operatorname{dim} E_{\lambda} \operatorname{ker}L: \lambda \in \mathbb{B}_{d}\},
    \end{eqnarray}
    where $L:H_d^2 \otimes\operatorname{ran}\,\Delta_{\mathcal{M}^\perp} \rightarrow \mathcal{M}^\perp$ is determined by $L(p\otimes \xi)=p\cdot \Delta_{\mathcal{M}^\perp} \xi$ for $p \in \mathcal{P}_d,~\xi \in \operatorname{ran}\,\Delta_{\mathcal{M}^\perp},$ and for $\lambda \in \mathbb{B}_d,$ $$E_{\lambda} \operatorname{ker}L=\{f(\lambda):f\in \operatorname{ker}L\}.$$
Assume that
\begin{equation}\label{(2.3)}
   \text{dim}\, \mathcal{M}\cap (\mathbb{C} \otimes \mathbb{C}^N)=n.
\end{equation}
 By Lemma \ref{dimdefect}, there exists an orthonormal basis $\{e_i\}_{i=1}^{N}$ of $\mathbb{C}^N$ such that
$\{P_{\mathcal{M}^\bot}(1\otimes e_i)\}_{i=n+1}^{N}$ and $\{1\otimes e_i\}_{i=1}^{n}$ are linear bases of $\operatorname{ran}\,\Delta_{\mathcal{M}^\perp}^2$ and $\mathcal{M}\cap (\mathbb{C} \otimes \mathbb{C}^N)$ respectively.
Take $f_i\in \mathcal{M}^\perp$ such that $\Delta_{\mathcal{M}^\perp}^2 f_i =P_{\mathcal{M}^\bot}(1\otimes e_i),$ $n+1 \leq i \leq N.$
Obviously, $\{\Delta_{\mathcal{M}^\perp} f_i\}_{i=n+1}^{N}$ is linearly independent.

Let $A:\operatorname{ran}\,\Delta_{\mathcal{M}^\perp} \rightarrow \mathbb{C}^{N-n}$ be a linear mapping determined by
$A(\Delta_{\mathcal{M}^\perp} f_i)=e_i,\, n+1 \leq i \leq N$ and set $B=I\otimes A: H_d^2  \otimes \operatorname{ran}\,\Delta_{\mathcal{M}^\perp}  \rightarrow H_d^2 \otimes \mathbb{C}^{N-n}.$
For any $ g_i \in H_d^2,$ there are sequences $\{p_m^i\}$ of polynomials such that $\lim\limits_{m\rightarrow \infty}||p_m^i-g_i||=0,$ and hence
$$\lim\limits_{m\rightarrow \infty}\left\|\sum\limits_{i=n+1}^{N} p_m^i \otimes e_i-\sum\limits_{i=n+1}^{N} g_i \otimes e_i\right\|=0.$$
Write $g=\sum\limits_{i=n+1}^{N} g_i \otimes e_i.$
Therefore
$$
\begin{aligned}
L(B^{-1}g)&=\lim\limits_{m\rightarrow \infty}L\left(\sum\limits_{i=n+1}^{N}  p_m^i \otimes \Delta_{\mathcal{M}^\perp}(f_i)\right)\\&=\lim\limits_{m\rightarrow \infty}\sum\limits_{i=n+1}^{N} p_m^i \cdot \Delta_{\mathcal{M}^\perp}^2(f_i)
\\&=\lim\limits_{m\rightarrow \infty}\sum\limits_{i=n+1}^{N} p_m^i \cdot P_{\mathcal{M}^\perp} (1 \otimes e_i)\\&= \lim\limits_{m\rightarrow \infty}\sum\limits_{i=n+1}^{N}  P_{\mathcal{M}^\perp}
( p_m^i \otimes e_i)\\&= P_{\mathcal{M}^\perp} g.
\end{aligned}
$$
Considering $\mathbb{C}^{N-n}$ as a subspace of $\mathbb{C}^{N},$
set $\mathcal{M}^{\prime}=\mathcal{M}\cap (H_d^2 \otimes \mathbb{C}^{N-n}).$
it is easy to see that $\text{ker} (L)=B^{-1} \mathcal{M}^{\prime}.$
 Notice that $\text{dim} E_\lambda \mathcal{M}^{\prime}= \text{dim} E_\lambda B^{-1}\mathcal{M}^{\prime},$ therefore
 \begin{eqnarray}\label{eq:jiu1}
  \text{dim}E_\lambda \text{ker} (L)=\text{dim} E_\lambda B^{-1}\mathcal{M}^{\prime}=\text{dim} E_\lambda \mathcal{M}^{\prime}.
   \end{eqnarray}
  From \eqref{(2.3)},
we have
$$\mathcal{M}=\mathcal{M}^{\prime}\oplus (H_d^2 \otimes \text{span}\{e_{1},\cdots,e_n\}).$$ Then
it is not difficult to see that $fd(\mathcal{M})=fd(\mathcal{M}^{\prime})+n.$ Hence
by \eqref{eq:K} and \eqref{eq:jiu1}, $$K(\mathcal{M}^\perp)=N-n-\sup \{\operatorname{dim} E_{\lambda} \mathcal{M}^{\prime}: \lambda \in \mathbb{B}_{d}\}=N-fd(\mathcal{M}).$$
\end{proof}
To distinguish between the linear independence of a subset in a module and the linear independence of a subset in a vector space, we provide the following notation.
\begin{Notation}
Let $\mathcal{A}$ be a module over an integral domain $\mathfrak{R},$ a finite subset $\{a_i\}_{i=1}^m\subseteq \mathcal{A}$ is called to be linearly independent \it{in the module $\mathcal{A},$} if
$$\sum\limits_{i=1}^mp_ia_i=0, \quad p_i \in \mathfrak{R}$$ implies that $p_i=0.$ A subset $\mathfrak{S}$ is a basis of the module $\mathcal{A}$ if $\mathfrak{S}$ is linearly independent in the module $\mathcal{A},$ and generates $\mathcal{A}.$
\end{Notation}
The following easy lemma should be well-known before, and we omit the proof.
\begin{lem}\label{lemF}
For a finitely generated module $F$ over $\mathcal{P}_d$, let $\mathfrak{S}$ be a generating set of $F$. Then the number of elements in any maximal linearly independent set of $\mathfrak{S}$
in the module $F$
is $\text{rank} \,(F)$, regardless of the choice of $\mathfrak{S}$.
\end{lem}
With the above preparations, we can prove our main result. At first, we restate Theorem \ref{thm1.1} as follows.
\begin{thm}\label{thm2.611}
If $\mathcal{M}\subseteq H_d^2 \otimes \mathbb{C}^N$ is a submodule,
then the
  following statements are equivalent.
\begin{itemize}
     \item[(1)] $K(\mathcal{M}^\bot)=\chi(\mathcal{M}^\bot)$;
     \item[(2)]
  there exists a nonempty open set $U\subseteq \operatorname{mp}(\mathcal{M})$ and polynomials $\{p_i\}_{i=1}^\mathfrak{m} \subseteq \mathcal{M}$ such that
  $$E_{\lambda}\mathcal{M}=\operatorname{span}\{p_i(\lambda):1\leq i\leq \mathfrak{m}\}$$ for all $\lambda \in U;$
     \item[(3)]$\mathcal{M}$ is locally algebraic.
\end{itemize}
\end{thm}
\begin{proof}(2)$\Rightarrow$(3): Obviously.

(3)$\Rightarrow$(1): 
For simplicity, set $l=fd(M).$ By the local algebraicity of $\mathcal{M},$ we can take $\lambda_{0}\in \operatorname{mp}(\mathcal{M})$ and polynomials $\{p_i\}_{i=1}^l \subseteq \mathcal{M}$ such that
 $$E_{\lambda_{0}}\mathcal{M}=\operatorname{span}\{p_i(\lambda_{0}):1\leq i\leq l\}.$$
Write $p_i=(p_{i,1},\cdots,p_{i,N}).$
Then
 \begin{eqnarray}\label{eq:2}
\text{rank} \left[
                 \begin{array}{cccc}
                   p_{1,1}(\lambda_{0}) & p_{1,2}(\lambda_{0}) & \cdots & p_{1,N}(\lambda_{0}) \\
                   p_{2,1}(\lambda_{0}) &p_{2,2}(\lambda_{0}) & \cdots &p_{2,N}(\lambda_{0}) \\
                   \vdots & \vdots & \ddots & \vdots \\
                    p_{l,1}(\lambda_{0}) & p_{l,2}(\lambda_{0}) & \cdots & p_{l,N}(\lambda_{0})\\
                 \end{array}
               \right]=l,
\end{eqnarray}
and hence there exist $1\leq i_1<\cdots<i_l\leq N$ such that the determinant
  \begin{eqnarray}\label{eq:3}
 \left|
                 \begin{array}{cccc}
                   p_{1,i_1}(\lambda_{0}) & p_{1,i_2}(\lambda_{0}) & \cdots & p_{1,i_l}(\lambda_{0}) \\
                   p_{2,i_1}(\lambda_{0}) &p_{2,i_2}(\lambda_{0}) & \cdots &p_{2,i_l}(\lambda_{0}) \\
                   \vdots & \vdots & \ddots & \vdots \\
                    p_{l,i_1}(\lambda_{0}) & p_{l,i_2}(\lambda_{0}) & \cdots & p_{l,i_l}(\lambda_{0})\\
                 \end{array}
               \right|\neq 0,
\end{eqnarray}
 which implies that the determinant
  \begin{eqnarray}\label{eq:4}
 \left|
                 \begin{array}{cccc}
                   p_{1,i_1} & p_{1,i_2} & \cdots & p_{1,i_l} \\
                   p_{2,i_1} &p_{2,i_2} & \cdots &p_{2,i_l} \\
                   \vdots & \vdots & \ddots & \vdots \\
                    p_{l,i_1} & p_{l,i_2} & \cdots & p_{l,i_l}\\
                 \end{array}
               \right|\neq 0.
\end{eqnarray}
By \cite[Page 216]{Arveson curvature}, Lemma \ref{KNf} and \eqref{oulashu}, $$K(\mathcal{M}^\bot)\leq \chi(\mathcal{M}^\bot), \quad K(\mathcal{M}^\bot)= N-fd(\mathcal{M})\quad \text{and}\quad \chi(\mathcal{M}^\bot)=\text{rank}(\mathbb{M}_{\mathcal{M}^\bot} ).$$ Therefore in order to obtain $K(\mathcal{M}^\bot)=\chi(\mathcal{M}^\bot),$
it suffices to show that
 \begin{eqnarray}\label{suffices}
 \text{rank}(\mathbb{M}_{\mathcal{M}^\bot} )\leq N-l.
\end{eqnarray}
By \eqref{(2.1)}, $\{P_{\mathcal{M}^\bot}(1\otimes e_{i})\}_{i=1}^{N}$ generates $\mathbb{M}_{\mathcal{M}^\bot} .$ By Lemma \ref{lemF},
it suffices to show that
for any $1\leq w_1<\cdots<w_{N-l+1}\leq N,$  $\{{P_{\mathcal{M}^\bot}(1\otimes e_{w_i})}\}_{i=1}^{N-l+1}$ is
  linearly dependent in the module $\mathbb{M}_{\mathcal{M}^\bot} .$ Indeed,
let
$$\{v_1,\cdots,v_{l-1}\}=\{1,\cdots,N\}\setminus \{w_1,\cdots,w_{N-l+1}\},$$ and
 \begin{eqnarray}\label{eq:6}
 C^{(k)}=\left[
                 \begin{array}{cccc}
                   p_{1,k} & p_{1,v_1} & \cdots & p_{1,v_{l-1}}\\
                   p_{2,k} &p_{2,v_1} & \cdots &p_{2,v_{l-1}} \\
                   \vdots & \vdots & \ddots & \vdots \\
                    p_{l,k} & p_{l,v_1} & \cdots & p_{l,v_{l-1}}\\
                 \end{array}
               \right], \quad 1\leq k\leq N.
\end{eqnarray}
Let $C^{(k)}_{ij}$ be the algebraic cofactor of of the $(i,j)$ element in $C^{(k)}.$
Notice that for all $1\leq i \leq l,$ $$C^{(1)}_{i1}=C^{(2)}_{i1}=\cdots=C^{(N)}_{i1},$$ and write
 $r_{i1}=C^{(k)}_{i1}.$ Next, we will prove that  $\{{P_{\mathcal{M}^\bot}(1\otimes e_{w_i})}\}_{i=1}^{N-l+1}$ is
  linearly dependent in the module $\mathbb{M}_{\mathcal{M}^\bot} $ in two cases.
  \begin{Case}
There exists an $i_0 \in \{1,\cdots,l\}$ such that $r_{i_01}=0.$
\end{Case}
There are polynomials $q_1, \cdots, q_{i_0-1}, q_{i_0+1}, \cdots, q_l,$ which are not all zero, such that
 \begin{eqnarray}\label{eq:10}
\left(\sum\limits_{i\neq i_0} q_i p_{i,v_1},\cdots,\sum\limits_{i\neq i_0} q_i p_{i,v_{l-1}}\right)=\sum\limits_{i\neq i_0} q_i \left(p_{i,v_1}, \ldots, p_{i,v_{l-1}}\right)=0.
\end{eqnarray}
By \eqref{eq:4}, $p_1, \cdots, p_{i_0-1}, p_{i_0+1}, \cdots, p_l$ are linearly independent in the module $\mathcal{P}_d^{N} $, which gives that
 \begin{eqnarray}\label{eq:11}
\sum\limits_{i\neq i_0} q_i p_i \neq 0.
\end{eqnarray}
Write $\left(f_1, \cdots,f_j,\cdots, f_N\right)=\sum\limits_{i\neq i_0} q_i p_i,$ then
\begin{equation}\label{eq:13}
\begin{aligned}
\sum\limits_{i\neq i_0} q_i p_i&=\sum\limits_{i\neq i_0} q_i(p_{i1},\cdots,p_{ij},\cdots,p_{iN})\\&=\left(\sum\limits_{i\neq i_0} q_ip_{i1},\cdots,\sum\limits_{i\neq i_0} q_ip_{ij},\cdots,\sum\limits_{i\neq i_0} q_ip_{iN}\right).
\end{aligned}
\end{equation}
 By \eqref{eq:10},
$
f_{v_i}=0 ~\text{for all} ~1 \leq i \leq l-1 ,
$
and
hence by \eqref{eq:11}, $(f_{w_1},\cdots,f_{w_{N-l+1}})\neq0.$
It follows from \eqref{eq:13} that
$$\sum\limits_{i=1}^{N-l+1} f_{w_i}\otimes e_{w_i}=\sum\limits_{i=1}^Nf_i\otimes e_i=\sum\limits_{i\neq i_0} q_i p_i\in \mathcal{M}.$$
 This ensures that
 $$\sum\limits_{i=1}^{N-l+1}f_{w_i}\cdot{P_{\mathcal{M}^\bot}(1\otimes e_{w_i})}=0.$$
 \begin{Case}
 For all $1 \leq i \leq l,$ $r_{i1}\neq 0.$
  \end{Case}
By \eqref{eq:4}, $\{p_i\}_{i=1}^{l}$ is linearly independent in the module $\mathcal{P}_d^{N}$,
hence
\begin{equation}\label{eq:8}
\begin{aligned}
\left (\left|C^{(1)}\right|,\cdots,\left|C^{(k)}\right|,\cdots,\left|C^{(N)}\right|\right)&=\left(\sum\limits_{i=1}^{l}r_{i1}p_{i1},\cdots,\sum\limits_{i=1}^{l}r_{i1}p_{ik},\cdots,\sum\limits_{i=1}^{l}r_{i1}p_{iN}\right)\\&=
\sum\limits_{i=1}^{l}r_{i1}\left(p_{i1},\cdots,p_{ik},\cdots,p_{iN}\right)\\&=\sum\limits_{i=1}^{l}r_{i1}p_i \neq 0.
\end{aligned}
\end{equation}
 By the definition of $C^{(k)},$
  \begin{eqnarray}\label{eq:9}
\left| C^{(v_1)}\right|=\cdots= \left|C^{(v_{l-1})}\right|=0.
\end{eqnarray}
Therefore $\left(\left|C^{(w_1)}\right|,\cdots,\left|C^{(w_{N-l+1})}\right|\right)\neq 0.$
 By \eqref{eq:8} and \eqref{eq:9},
$$\sum\limits_{i=1}^{N-l+1} \left|C^{(w_i)}\right|\otimes e_{w_i}=\sum\limits_{i=1}^N\left|C^{(i)}\right|\otimes e_i=\sum\limits_{i=1}^{l}r_{i1}p_i\in \mathcal{M}.$$
 This ensures that
 $$\sum\limits_{i=1}^{N-l+1}\left|C_{w_i}\right|\cdot{P_{\mathcal{M}^\bot}(1\otimes e_{w_i})}=0.$$

(1)$\Rightarrow$(2): Suppose that $K(\mathcal{M}^\bot)=\chi(\mathcal{M}^\bot).$ Set $l=fd(\mathcal{M}),$ by Lemma \ref{KNf},
$$\chi(\mathcal{M}^\bot)=N-l.$$ By \eqref{oulashu} and Lemma \ref{lemF}, each maximal linearly independent set of $\{{P_{\mathcal{M}^\bot}(1\otimes e_{i})}\}_{i=1}^{N}$
in the module $\mathbb{M}_{\mathcal{M}^\bot} $
has $N-l$ elements.
For any maximal linearly independent set $\{{P_{\mathcal{M}^\bot}(1\otimes e_{i_k})}\}_{k=1}^{N-l}$
in the module $\mathbb{M}_{\mathcal{M}^\bot},$
 write $$\{\alpha_1,\cdots,\alpha_{l}\}=\{1,\cdots,N\}\setminus \{i_1,\cdots,i_{N-l}\}.$$
Then for any fixed $1\leq s\leq l$, there are polynomials $p_{i,s}, ~1\leq i \leq N-l+1,$ which are not all zero,
such that $$\sum\limits_{j=1}^{N-l}p_{j,s} \otimes e_{i_{j}}+p_{N-l+1,s} \otimes e_{\alpha_s}\in \mathcal{M}.$$
Since $\{{P_{\mathcal{M}^\bot}(1\otimes e_{i_k})}\}_{k=1}^{N-l}$ is linearly independent
in the module $\mathbb{M}_{\mathcal{M}^\bot},$ we have $p_{N-l+1,s}$ $\neq 0$ for all $1\leq s\leq l.$
Let \begin{equation}\label{2.14}
U = \left(\bigcap_{s=1}^{l} Z(p_{N-l+1,s})^c\right) \bigcap \mathbb{B}_d,
\end{equation}
 and
 $$Q_s=\sum\limits_{j=1}^{N-l}p_{j,s} \otimes e_{i_{j}}+p_{N-l+1,s} \otimes e_{\alpha_s}\in \mathcal{M}, \quad 1\leq s\leq l.$$
We claim that $Q_1(\lambda),\cdots, Q_l(\lambda)$ are linearly independent for every $\lambda \in U$. Indeed, if
$$\sum\limits_{s=1}^{l} \mu_s Q_s(\lambda)=0,\quad \mu_s \in \mathbb{C},$$
then $$\begin{aligned}
\sum\limits_{j=1}^{N-l}& \left(\sum\limits_{s=1}^{l}\mu_sp_{j,s}(\lambda)\right) e_{i_{j}}+\sum\limits_{s=1}^{l}\mu_sp_{N-l+1,s}(\lambda) e_{\alpha_s}\\&=\sum\limits_{s=1}^{l}\mu_s\left(\sum\limits_{j=1}^{N-l}p_{j,s} \otimes e_{i_{j}}\right)(\lambda)+\sum\limits_{s=1}^{l}\mu_s(p_{N-l+1,s} \otimes e_{\alpha_s})(\lambda)\\&=\sum\limits_{s=1}^{l}\mu_sQ_s(\lambda)=0.
\end{aligned}$$
Hence
\begin{align}\label{ks}
\sum\limits_{s=1}^{l}\mu_sp_{N-l+1,s}(\lambda) e_{\alpha_s}=0.
\end{align}
It follows from \eqref{2.14} and (\ref{ks}) that $\mu_s=0$ for all $1\leq s\leq l,$ and the claim is proved.
Obviously, for every $\lambda \in U,$
$$fd(M)=l=\operatorname{dim}\operatorname{span}\{Q_i(\lambda):1\leq i\leq l\}.$$
Therefore $U\subseteq \text{mp}(\mathcal{M})$ and
$$E_{\lambda}\mathcal{M}=\operatorname{span}\{Q_i(\lambda):1\leq i\leq l\}, \quad \forall \lambda \in U.$$
\end{proof}

As mentioned in \eqref{oulashu}, $\chi(\mathcal{H}) = \operatorname{rank}(\mathbb{M}_{\mathcal{H}}).$ It can be seen that
$\operatorname{rank}(\mathbb{M}_{\mathcal{H}})$ plays a crucial role in proving Arveson's version of the Gauss-Bonnet-Chern formula as presented in the proof of Theorem \ref{thm2.611}.

For $N\geq 2,$ it is clear that there exist numerous locally algebraic submodules that are not polynomially generated. Consequently, by Theorem \ref{thm1.1}, we have that $K(\mathcal{M}^\bot) = \chi(\mathcal{M}^\bot).$
 To illustrate this point, we present the following example.
\begin{exam}
 Let $B(z_2)$ be an infinite Blaschke prodcut with zeros $\{\alpha_j\}.$
 Set $$\mathcal{M}=[z_1\otimes e_1,z_2\otimes e_2, B(z_2)\otimes e_1].$$
It can be verified that $\mathcal{M}$ is not polynomially generated while $\mathcal{M}$ is locally algebraic.
\end{exam}

\begin{rem}
In \cite[Page 214]{Arveson curvature}, Arveson established the following formula for calculating the Euler characteristic of a \( d \)-contractive Hilbert module \( \mathcal{H} \) of finite rank:
\begin{equation}\label{oula}
\chi(\mathcal{H}) = c(\mathbb{M}_{\mathcal{H}}) =d! \lim_{k \rightarrow \infty} \frac{\operatorname{dim} \mathbb{M}_k}{k^d},
\end{equation}
where
\[
\mathbb{M}_k = \operatorname{span}\{ f \cdot \zeta : f \in \mathcal{P}_d, \operatorname{deg} f \leq k, \zeta \in \operatorname{ran} \Delta_\mathcal{H} \}.
\]
Here, we provide a more fundamental approach to proving \eqref{oula} by the $\operatorname{rank}(\mathbb{M}_{\mathcal{H}})$. This discussion is presented in Proposition \ref{fulu1} in Appendix.
Moreover, for the quotient module $\mathcal{M}^\perp$
  of $H_d^2 \otimes \mathbb{C}^N$, we find that calculating the Euler characteristic by using $\operatorname{rank}(\mathbb{M}_{\mathcal{M}^\perp})$
 is more straightforward than applying the formula in \eqref{oula}. At the end of this section, we will compute the Euler characteristics of several examples from \cite[Proposition 7.14]{Arveson curvature} 
 {\color{red}by using this rank-based approach.}
 \end{rem}
 \begin{exam}

 Let $\phi$ be a homogeneous polynomial of degree at least 1 in $\mathcal{P}_d$, and define
$$
\mathcal{M}=\left\{(f, \phi \cdot f): f \in H_d^2\right\} \subseteq H_d^2 \otimes \mathbb{C}^2 .
$$
By Lemma \ref{dimdefect}, it is easy to see that $\{P_{\mathcal{M}^\perp}(1 \otimes e_i)\}_{i=1}^{2}$ is a linear basis of $\operatorname{ran}\,\Delta_{\mathcal{M}^\perp}.$
We claim that $\{P_{\mathcal{M}^\perp}(1 \otimes e_1)\}$ is a maximal linearly independent set of $\{P_{\mathcal{M}^\perp}(1 \otimes e_i)\}_{i=1}^{2}$
in the module $\mathbb{M}_{\mathcal{M}^\bot},$ which implies that  $\chi(\mathcal{M}^\perp)=1$ by \eqref{oulashu} and Lemma \ref{lemF}.
In fact, if $p\cdot P_{\mathcal{M}^\perp}(1 \otimes e_1) =0,$ then $p\otimes e_1\in \mathcal{M},$ hence by the definition of $\mathcal{M},$ we have $p=0.$  Moreover,
since $$1\cdot P_{\mathcal{M}^\perp}(1 \otimes e_1)+ \phi \cdot P_{\mathcal{M}^\perp}(1 \otimes e_2)=0,$$ then $\{P_{\mathcal{M}^\perp}(1 \otimes e_i)\}_{i=1}^{2}$ is linearly dependent
in the module $\mathbb{M}_{\mathcal{M}^\bot}.$
 \end{exam}

At the end of this section, we pose the following problem.
\begin{Problem}
  {\color{red}The Euler characteristic over the polynomial ring depends on the ring. Under what conditions does Arveson's version of the GBC formula for Hilbert modules over other rings hold? For example, for the Hilbert modules over the ring of multipliers?}
\end{Problem}

\section{Drury-Arveson space in infinitely many variables }

In this section, we will introduce some elementary properties of {\color{red}Drury-Arveson space} of infinitely many variables.
For a complex separable Hilbert space $\cal H$, let $\Gamma_s^n(\cal H)$ be the  symmetric $n$-fold tensor product of $\cal H$ with itself, and the symmetric Fock space
$$
\Gamma_s({\cal H})=\bigoplus\limits_{n=0}^\infty \Gamma_s^n(\cal H),
$$
where $\Gamma_0(\cal H)=\mathbb C$ is the vacuum space. In the historical literature \cite{Parthasarathy}, the symmetric Fock space describes the bosonic fields.

 By using the framework of Arveson \cite{Arveson 3}, for $m=\dim \cal H<\infty$, the symmetric Fock space $\Gamma_s(\cal H)$ is unitarily equivalent to $H_m^2$. For $\cal H$ being of infinite dimension, similar to the framework of Arveson\cite{Arveson 3}, $\Gamma_s(H)$ can be described as a Hilbert space of holomorphic functions on the unit ball $\mathbb B$ in $\ell^2$. Recall that a complex-valued function $g$ on an open subset $U$ of a Banach space $X$ is analytic on $U$ if $g$ is locally bounded, and for each $x_{0} \in U$ and any $x \in X$, the function $g\left(x_{0}+z x\right)$ depends analytically on the parameter $z$ for $x_{0}+z x \in U$.

It is easy to see that $H_m^2\subset H_{m+1}^2$ and for any $f\in H_m^2$,
$$
\|f\|_{H_m^2}=\|f\|_{H_{m+1}^2},
$$
which implies that $\bigcup\limits_{m=1}^\infty H_m^2$ can be equipped with a natural norm, and in what follows, we will denote its completion by $H^2$. Further,  any $f\in \bigcup\limits_{m=1}^\infty H_m^2$ can be seen as a holomorphic function on the unit ball $\mathbb B$ by
$$\hat{f}(z)=f(z_1,\cdots,z_m)\quad \text{for}~ f\in H_m^2, \quad z\in \mathbb B.$$
Next for any $f\in \bigcup\limits_{m=1}^\infty H_m^2$, assume $f\in H_{m_0}^2$ for $m_0>0$, and we have for any $\lambda\in\mathbb B$,
$$
|\hat{f}(\lambda)|=|f(\lambda_1,\cdots,\lambda_{m_0})| \leq \frac{\|f\|_{H^2_{m_0}}}{\sqrt{1-\sum\limits_{i=1}^{m_0}|\lambda_i|^2}} \leq \frac{\|f\|_{H^2}}{\sqrt{1-\|\lambda\|_{{\color{red}\ell^2}}^{2}}}.
$$
It follows that for any $\lambda\in\mathbb B$, the evaluation functional $E_{\lambda}$ on $\bigcup\limits_{m=1}^\infty H_m^2$, defined by $E_\lambda(f)=f(\lambda)$, is continuous, which is naturally extended to a continuous functional on  $H^2$. Therefore, as mentioned in the introduction, $H^2$ is a reproducing kernel Hilbert space of holomorphic functions on
 $\mathbb B$, with the reproducing kernel
$$
\mathcal{K}_{\lambda}(z)={1\over 1-\langle z,\lambda\rangle},
$$
which will be called Drury-Arveson space in infinitely many variables. Obviously, $H^2$ is unitarily equivalent to $\Gamma_s(\ell^2)$.
Let
$$\mathbb{N}^{(\mathbb{N})} = \left\{ (\alpha_n)_{n \in \mathbb{N}} \mid \alpha_n \in \mathbb{N} \text{ for all } n \in \mathbb{N}, \text{ and } \alpha_n = 0 \text{ for all but finitely many } n \right\}.$$
 For any $\alpha = (\alpha_1, \alpha_2, \cdots, \alpha_n, 0, 0, \cdots) \in \mathbb{N}^{(\mathbb{N})}$, set
 $$|\alpha|=\sum\limits_{i=1}^{n}\alpha_i, \quad \alpha!=\alpha_1! \cdots \alpha_n!,$$
 and $$z^\alpha = z_1^{\alpha_1} z_2^{\alpha_2} \cdots z_n^{\alpha_n}, \quad z \in \mathbb{B}.$$
It is straightforward to verify that the set $\{z^\alpha : \alpha \in \mathbb{N}^{(\mathbb{N})}\}$
is an orthogonal basis in $H^2$. Additionally, we have $$\|z^\alpha\|^2 = \frac{\alpha!}{|\alpha|!}.$$
As usual, write $M_i=M_{z_i}$ and it is obvious that \[
M_{i}^* z_1^{\alpha_1} \cdots z_n^{\alpha_n} =
\begin{cases}
0, & \text{if } \alpha_i = 0, \\
\frac{\alpha_i}{\alpha_1 + \cdots + \alpha_n} z_1^{\alpha_1} \cdots z_i^{\alpha_i - 1} \cdots z_n^{\alpha_n}, & \text{if } \alpha_i > 0.
\end{cases}
\]
The following proposition is trivial.
\begin{prop} Let $E_{0}$ be the orthogonal projection from $H^{2}$ onto the one-dimensional space of constant functions. Then $\sum\limits_{k=1}^m M_k M_{k}^*$ converges to $I-E_0$ in the strong operator topology as $m$ tends to infinity.
\end{prop}


The concept of inner sequence  was introduced by Arveson \cite{Arveson curvature} to study whether Arveson's curvature invariant $K(\cal M^\perp)$ is an integer or not. Recall that a sequence of multipliers $\{\phi_n\}$ {\color{red}in $H_d^2$} is called an inner sequence, if
$$
\sum\limits_{n=1}^\infty |\phi_n(\xi)|^2=1,\quad a.e. \text{ on }\partial \mathbb B_d.
$$
The existence problem of inner sequence, raised by Arveson in \cite[Page 225]{Arveson curvature}, was solved affirmatively by Greene, Richter and Sundberg\cite{inner multipliers}. Combining the ideas in \cite{inner multipliers} and the asymptotic techniques, the existence of inner sequence in $H^2$ also has an affirmative answer, which poses no essential difficulties. For the reader's convenience, we list the process below.

 Let $\mathbb{T}$ be an $\omega$-contraction on a Hilbert space $\mathcal{H}$ of finite rank.
For a submodule $\mathscr{M}$ of $H^2\otimes \text{ran}\,\Delta_\mathcal{H},$ by a similar proof to that of \cite[Page 191]{Arveson curvature}, there is a Hilbert space $E$ and a homomorphism of Hilbert module $\Phi:
H^2\otimes E\rightarrow H^2\otimes \text{ran}\,\Delta_\mathcal{H}$ such that
$$P_{\mathscr{M}}=\Phi \Phi^*.$$
By applying the Riesz lemma, we obtain a unique operator-valued function $z \mapsto \Phi(z) \in \mathcal{B}(E, \text{ran}\,\Delta_\mathcal{H})$
  that satisfies the following relation:
$$\left\langle \Phi(z) \zeta_1, \zeta_2 \right\rangle = \left\langle \Phi(1 \otimes \zeta_1), \mathcal{K}_z \otimes \zeta_2 \right\rangle, \quad \zeta_1 \in E, \quad \zeta_2 \in \text{ran}\,\Delta_\mathcal{H}, \quad z \in \mathbb{B}.$$
It is well known that any function \( f \in H^2(\mathbb{B}_m) \) has a boundary limit \( \widetilde{f} \) almost everywhere on \( \partial \mathbb{B}_m \). For the sake of clarity, we will not distinguish between \( f \) and \( \widetilde{f} \) throughout this paper.
Next, we will prove for almost every  $z^{(m)}\in \partial \mathbb{B}_{m}$, $\Phi(z^{(m)})$ is a partial isometry with $$\operatorname{rank} \Phi(z^{(m)})=\sup \left\{\operatorname{rank} \Phi(\lambda^{(m)}): \lambda^{(m)} \in\mathbb{ B}_{m}\right\}.$$ To do this, we need several lemmas, whose proofs are similar to that of \cite[Lemma 2.2, Theorem 4.3, Proposition 2.4 and Lemma 3.1]{inner multipliers}, and we list these lammas without proofs as follows.
\begin{lem}\label{Klambda}
For each $\lambda \in \mathbb{B},$ $\mathcal{K}_{\lambda} f \in H^2$ for any $f\in H^2.$
\end{lem}
\begin{lem}\label{f}
 For $f \in  H^2\otimes \emph{ran}\,\Delta_\mathcal{H},$
$$
\text{K-} \lim\limits_{\lambda^{(m)}\rightarrow z^{(m)}} \frac{\left\|f \mathcal{K}_{\lambda^{(m)}}\right\|}{\left\|\mathcal{K}_{\lambda^{(m)}}\right\|}=\|f(z^{(m)})\| \text { for } \sigma_m ~a.e.~ z^{(m)}=(z_1,\cdots,z_m) \in \partial \mathbb{B}_{m} \text {, }
$$
where $\text{K-} \lim\limits_{\lambda^{(m)}\rightarrow z^{(m)}} \frac{\left\|f \mathcal{K}_{\lambda^{(m)}}\right\|}{\left\|\mathcal{K}_{\lambda^{(m)}}\right\|}$ denotes the $K$-limit \cite[Page 319]{inner multipliers} of $\frac{\left\|f \mathcal{K}_{\lambda^{(m)}}\right\|}{\left\|\mathcal{K}_{\lambda^{(m)}}\right\|}$ at $z^{(m)}\in \partial \mathbb{B}_{m}.$
\end{lem}

\begin{lem}\label{rank}
Set
$$
\mathfrak{m}_m=\sup \left\{\operatorname{rank} \Phi(\lambda^{(m)}): \lambda^{(m)} \in \mathbb{B}_{m}\right\} ,
$$
then $\operatorname{rank} \Phi(\lambda^{(m)})=\mathfrak{m}_m$ on $\lambda^{(m)} \in \mathbb{B}_{m} \backslash \mathcal{E}_m$, where $\mathcal{E}_m$ is at most a countable union of zero varieties of nonzero bounded analytic functions in $\mathbb{B}_{m}$ and $\operatorname{rank} \Phi(z^{(m)})=\mathfrak{m}_m$ for $\sigma_m$ a.e. $z^{(m)}\in \partial \mathbb{B}_{m}$.
\end{lem}

For $\lambda^{(m)}=(\lambda_1,\cdots,\lambda_m) \in \mathbb{B}_{m}$, it is easy to see that
\begin{align}\label{zuihouyici}
\text{ran}\, \Phi(\lambda^{(m)})&=\{\Phi(\lambda^{(m)}) y: y \in E\}\nonumber\\&=\{\Phi(\lambda^{(m)}) f(\lambda): f \in H^2\otimes E\}\nonumber\\&=\left\{E_{\lambda^{(m)}}(\Phi f): f \in H^2\otimes E\right\}\nonumber\\&=E_{\lambda^{(m)}} \mathscr{M}.
\end{align}
The following proposition is inspired by \cite[Theorem 4.3]{inner multipliers}. For the reader's convenience, we will list its proof.
\begin{prop}\label{KH}Under the above assumptions,  for $\sigma_m$ a.e. $z^{(m)} \in \partial \mathbb{B}_{m}$, $\Phi(z^{(m)})$ is a partial isometry with $\operatorname{rank} \Phi(z^{(m)})=\mathfrak{m}_m.$
\end{prop}
\begin{proof} Assume $z^{(m)} \in$ $\partial \mathbb{B}_{m}$  such that the $K$-limits of $\Phi(\lambda^{(m)})$ and $\Phi(\lambda^{(m)})^{*}$ exist at $z^{(m)}$ in the strong operator topology. We have to show that for a.e. $z^{(m)} \in \partial \mathbb{B}_{m},~ \Phi(z^{(m)})^{*}$ is an isometry on $\operatorname{ran} \Phi(z^{(m)})$. Since $\left\|\Phi(z^{(m)})^{*}\right\| \leqslant 1,$ it suffices to show that for all $y \in E$ with $\Phi(z^{(m)}) y \neq 0$
\begin{equation}\label{mudi}
\| \Phi(z^{(m)})^{*} \Phi(z^{(m)}) y \|_{E} \geq \| \Phi(z^{(m)}) y \|.
\end{equation}

Note that for $\lambda^{(m)}=(\lambda_1,\cdots,\lambda_m) \in \mathbb{B}_{m}$ and $x \in \text{ran}\,\Delta_\mathcal{H}$ we have
$$
\left\|P_{\mathscr{M}} \left(\mathcal{K}_{\lambda^{(m)}} \otimes x\right)\right\|=\sup \left\{\left|\langle f(\lambda^{(m)}), x\rangle\right|: f \in \mathscr{M},\|f\| \leqslant 1\right\},
$$
because for $f \in \mathscr{M},$  $$\left|\langle f(\lambda^{(m)}), x\rangle\right|=\left|\left\langle f, \mathcal{K}_{\lambda^{(m)}} \otimes x\right\rangle\right|=\left|\left\langle f, P_{\mathscr{M}} \left(\mathcal{K}_{\lambda^{(m)}} \otimes x\right)\right\rangle\right| \leqslant\|f\|\left\|P_{\mathscr{M}} \left(\mathcal{K}_{\lambda^{(m)}} \otimes x\right)\right\|$$
with equality if $f=P_{\mathscr{M}} \left(\mathcal{K}_{\lambda^{(m)}} \otimes x\right).$ Hence by Lemma \ref{Klambda},
for nonzero $f \in \mathscr{M}$, $\lambda^{(m)} \in \mathbb{B}_{m}$, and $x \in \text{ran}\,\Delta_\mathcal{H},$
\begin{equation}\label{mudi1}
\begin{aligned}
\left\|\Phi(\lambda^{(m)})^{*} x\right\|&=\frac{\left\|\mathcal{K}_{\lambda^{(m)}}\otimes \Phi(\lambda^{(m)})^{*} x\right\|}{\left\|\mathcal{K}_{\lambda^{(m)}}\right\|}\\&=\frac{\left\|\Phi^{*}\left(\mathcal{K}_{\lambda^{(m)}}\otimes  x\right)\right\|}{\left\|\mathcal{K}_{\lambda^{(m)}}\right\|}\\&=\frac{\left\|P_{\mathscr{M}}\left(\mathcal{K}_{\lambda^{(m)}}\otimes x\right)\right\|}{\left\|\mathcal{K}_{\lambda^{(m)}}\right\|} \\
 & \geqslant\frac{\left|\left\langle\left(\mathcal{K}_{\lambda^{(m)}} f\right)(\lambda^{(m)}), x\right\rangle\right|}{\left\|\mathcal{K}_{\lambda^{(m)}}\right\|\left\|\mathcal{K}_{\lambda^{(m)}} f\right\|}\\&=\frac{\left|\langle f(\lambda^{(m)}), x\rangle\right|}{\left\|\mathcal{K}_{\lambda^{(m)}} f\right\| /\left\|\mathcal{K}_{\lambda^{(m)}}\right\|} .
 \end{aligned}
\end{equation}

Now for $y \in E$ with $\Phi(z^{(m)}) y \neq 0,$ set $f=\Phi (1\otimes y)$ and $x=\Phi(z^{(m)}) y.$ Then Lemma \ref{f} and \eqref{mudi1} ensure \eqref{mudi}.
\end{proof}

\section{Arveson's version of the Gauss-Bonnet-Chern formula in infinitely many variables}

In the present section, we will establish the asymptotic curvature invariant and the asymptotic Euler characteristic of finite rank $\omega$-contractive Hilbert modules.
 For a Hilbert space $\mathcal{H}$, recall that an $\omega$-contraction is an infinite tuple of commuting operators $\mathbb{T}=\left(T_{1}, \cdots, T_{m}, \cdots\right)$ acting on $\mathcal{H},$
 satisfying
\begin{eqnarray}\label{3.1}
\label{infinite con}T_{1}T_{1}^{*}+\cdots+T_{m}T_{m}^{*}\leq I ~\operatorname{for} ~\operatorname{all} ~m>0.
\end{eqnarray}

By \cite[Remark 3.2]{Arveson 3}, an infinite tuple $\mathbb{T}$ of commuting operators is an $\omega$-contraction
iff for all positive integers  $m$ and $\{\zeta_i\}_{i=1}^{m} \subseteq \cal H$,
\begin{eqnarray}\label{3.210}
\left\|\sum\limits_{i=1}^m T_{i}\zeta_{i}\right\|^2\leq \sum\limits_{i=1}^m||\zeta_{i}||^2.
\end{eqnarray}
To simplify the notation. Set \begin{eqnarray}\label{RH}
R=\overline{\text{ran}\,\Delta_\mathcal{H}}
\end{eqnarray}
temporarily.
For $\omega$-contractive Hilbert modules, we have the following result.
\begin{thm}\label{2.6} Let $\mathbb{T}$ be an $\omega$-contraction on a Hilbert space $\mathcal{H}.$
 Then for every orthonormal basis $e_{1}, \cdots, e_{m}, \cdots$ of $\ell^2,$ there is a unique bounded operator $$L: \Gamma_{s}(\ell^2) \otimes R \rightarrow \mathcal{H}$$ satisfying
\begin{eqnarray}\label{3.4}
L(1 \otimes \xi)=\Delta_\mathcal{H} \xi,~\text{and}~~
L\left(e_{i_{1}} e_{i_{2}} \cdots e_{i_{n}} \otimes \xi\right)=T_{i_{1}} T_{i_{2}} \cdots T_{i_{n}} \Delta_\mathcal{H} \xi, \quad i_{1}, \cdots, i_{n} \in \mathbb{Z}^+,
\end{eqnarray}
where
$e_{i_{1}} e_{i_{2}} \cdots e_{i_{n}}$ is the symmetric tensor product of $e_{i_1},\cdots, e_{i_n}$.
 In addition, $\|L\| \leq 1.$ Moreover, if\, $\mathbb{T}$ is pure, then $L$ is a coisometry: $L L^{*}=I_{\mathcal{H}}$.
\end{thm}
\begin{proof}
For every $\eta \in \mathcal{H}$,
 $$
\begin{aligned}
&\sum_{i_{1}=1}^{\infty}\cdots \sum_{ i_{n}=1}^{\infty}||e_{i_{1}} \otimes \cdots \otimes e_{i_{n}}||^{2}
\langle\Delta_\mathcal{H} T_{i_{n}}^{*} \cdots T_{i_{1}}^{*} \eta, \Delta_\mathcal{H} T_{i_{n}}^{*} \cdots T_{i_{1}}^{*} \eta\rangle
\\&=\sum_{i_{1}=1}^{\infty}\cdots \sum_{ i_{n}=1}^{\infty}\langle T_{i_{1}} \cdots T_{i_{n}}\Delta_\mathcal{H} ^2T_{i_{n}}^{*} \cdots T_{i_{1}}^{*} \eta, \eta\rangle\\&=\langle\phi^n(I)\eta,\eta\rangle-\langle\phi^{n+1}(I)\eta,\eta\rangle\\&\leq
||\eta||^{2}.
\end{aligned}
$$
Let $$
\zeta_{n}=\sum_{i_{1}=1}^{\infty}\cdots \sum_{ i_{n}=1}^{\infty} e_{i_{1}} \otimes \cdots \otimes e_{i_{n}} \otimes \Delta_\mathcal{H} T_{i_{n}}^{*} \cdots T_{i_{1}}^{*} \eta,
$$
then $\zeta_{n}\in {\ell^2}^{\otimes n}.$ The commutativity of $\{T_i^*\}_{i=1}^{\infty}$ implies that
$\zeta_{n}\in \Gamma_{s}^{n}(\ell^2).$
Moreover,
\begin{eqnarray}\label{3.3}
\sum_{n=0}^{\infty}\left\|\zeta_{n}\right\|^{2}=||\eta||^2-\lim\limits_{n\rightarrow\infty}\langle \phi^{n+1}(I) \eta, \eta\rangle\leq ||\eta||^2,
\end{eqnarray}
we have $(\zeta_{0}, \zeta_{1}, \zeta_{2}, \cdots)\in  \Gamma_{s}(\ell^2) \otimes R.$
Define $D:\mathcal{H}\rightarrow\Gamma_{s}(\ell^2) \otimes R$ by $$D(\eta)=(\zeta_{0}, \zeta_{1}, \zeta_{2}, \cdots),$$
    and let $L=D^*.$
By (\ref{3.3}), $||L||\leq1.$ Furthermore, if $\mathbb{T}$ is pure, then
$$\lim\limits_{n\rightarrow \infty}\langle\phi^n(I)\eta,\eta\rangle=0,$$
which implies that $L$ is a coisometry.

Next, we will check \eqref{3.4}. For $\zeta=1 \otimes \xi$ with $\xi \in R,$
$$
\langle 1 \otimes \xi, D \eta\rangle=\langle 1 \otimes \xi, 1 \otimes \Delta \eta\rangle=\langle\Delta \xi, \eta\rangle,
$$
and for $\zeta=e_{j_{1}} \cdots e_{j_{n}} \otimes \xi$ for $n \geq 1, j_{1}, \cdots, j_{n} \in \mathbb{Z}^+$ and $\xi \in R,$
$$
\begin{aligned}
\langle L(\zeta), \eta\rangle&=\left\langle e_{j_{1}} \cdots e_{j_{n}} \otimes \xi, D \eta\right\rangle \\&=\left\langle e_{j_{1}} \otimes\cdots \otimes e_{j_{n}} \otimes \xi, D \eta\right\rangle \\&=\sum_{i_{1}=1}^{\infty}\cdots \sum_{ i_{n}=1}^{\infty}\left\langle e_{j_{1}} \cdots e_{j_{n}} \otimes \xi, e_{i_{1}} \otimes \cdots \otimes e_{i_{n}} \otimes \Delta_\mathcal{H} T_{i_{n}}^{*} \cdots T_{i_{1}}^{*} \eta\right\rangle
\\&=\left\langle T_{j_{1}} \cdots T_{j_{n}} \Delta_\mathcal{H} \xi, \eta\right\rangle.
\end{aligned}
$$
At last, please note $\{1, e_{i_1},e_{i_2}e_{i_3},e_{i_4}e_{i_5}e_{i_6},\cdots :i_k \geq 1, k \geq 1\}$ is an orthogonal basis in $\Gamma_s(\ell^2),$ and we obtain the uniqueness of $L.$
\end{proof}

The proof of the following Lemma is almost same as \cite[Lemma 1.8]{Arveson curvature}, and for the convenience of the readers,
 we write it down here without proof.

\begin{lem}\label{lem2.7}
Let $\mathcal{H}$ be an $\omega$-contractive Hilbert $\mathcal{P}_\infty$-module. Let $L: H^{2} \otimes R \rightarrow \mathcal{H}$ be the operator defined in Theorem \ref{2.6}. Then there is a Hilbert $\mathcal{P}_\infty$-module $E$ and a homomorphism $\Phi \in \operatorname{hom}\left(E, H^{2} \otimes R \right)$ such that
$$
L^{*} L+\Phi \Phi^{*}=I_{H^{2} \otimes R} .
$$
\end{lem}
Let $\mathcal{H}$ be an $\omega$-contractive Hilbert module of finite rank. For $z=(z_{1}, \cdots, z_{m},\cdots)\in \ell^2,$
set $S_m(z)=\bar{z}_{1} T_{1}+\cdots+\bar{z}_{m} T_{m}.$ {\color{red}Then
for $\xi\in \mathcal{H},$ by \eqref{3.210}, we have $$\|S_m(z)\xi-S_{m+p}(z)\xi\|^2=\|\bar{z}_{m} T_{m}\xi+\cdots+\bar{z}_{m+p} T_{m+p}\xi\|^2\leq\sum\limits_{i=m}^{m+p}||z_{i}\xi||^{2}=\sum\limits_{i=m}^{m+p}|z_{i}|^{2}||\xi||^{2},$$
 which shows that $S_m(z)\xi$ is a Cauchy sequence. Hence 
$S_m(z)$ converges in the strong operator topology.} Let
$$
T(z)=\text{SOT-}\lim\limits_{m\rightarrow\infty}S_m(z),
$$
therefore
$$\|T(z)\xi\|^2=\lim\limits_{m\rightarrow\infty}||S_m(z)\xi||^2\leq\lim\limits_{m\rightarrow\infty}\sum \limits_{i=1}^{m}||z_{i}\xi||^{2}=||z||_{\ell^2}^2||\xi||^2,
$$
$i.e.~\|T(z)\|\leq ||z||_{\ell^2}.$
In particular, for $z\in \mathbb{B},$
$\|T(z)\|<1,$ and hence $I-T(z)$ is invertible. Define a positive-operator-valued function $F(z)$ on $\mathbb{B}$ by
\begin{eqnarray}
\label{F}F(z) \xi=\Delta\left(I-T(z)^{*}\right)^{-1}(I-T(z))^{-1} \Delta \xi, \quad \xi \in \Delta_ \mathcal{H}.
\end{eqnarray}

By similar proofs to that of \cite[Theorem 1.2 and Theorem A]{Arveson curvature}, we have
\begin{lem} \label{thm5.9} Let $F: \mathbb{B} \rightarrow \mathcal{B}(R)$ be defined in \eqref{F}. There is a Hilbert space $E$ and an operator-valued holomorphic function $\Phi: \mathbb{B} \rightarrow \mathcal{B}(E, R)$ such that
\begin{equation}\label{mudi2}
\begin{aligned}
\left(1-|z|^{2}\right) F(z)=I-\Phi(z) \Phi(z)^{*}, \quad z \in \mathbb{B} .
\end{aligned}
\end{equation}
\end{lem}
 \begin{lem}\label{eq11} For an $\omega$-contractive Hilbert module $\mathcal{H}$ of finite rank, let $F: \mathbb{B} \rightarrow \mathcal{B}(R)$ be defined by \eqref{F}. Then for $\sigma_m$-almost every $z^{(m)} \in \partial \mathbb{B}_{m}$, the limit
 \begin{equation}\label{mudi3}
\begin{aligned}
K_{0}^m(z^{(m)})=\lim _{r \rightarrow 1}\left(1-r^{2}\right) \operatorname{trace} F(r z_1,\cdots,rz_m,0,0,\cdots)
\end{aligned}
\end{equation}
exists and satisfies
$$
0 \leqq K_{0}^m(z^{(m)}) \leqq \operatorname{rank} \mathcal{H}.
$$
 \end{lem}

Set \begin{equation}\label{fulu4.9}
K_m(\mathcal{H})=\int_{\partial \mathbb{B}_{m}} K_{0}^m(z^{(m)}) d \sigma_m.
\end{equation}
 To define the asymptotic curvature invariant,
we need the following technical lemma.


\begin{lem} \label{lem3.6} If $\mathcal{H}$ is a pure $\omega$-contractive Hilbert module of finite rank, then $K_{m}(\mathcal{H})$ is an integer, and
$$
\begin{aligned}
K_{m}(\mathcal{H})&=\operatorname{rank} \mathcal{H}-\sup \{\operatorname{dim} E_{\lambda} \mathscr{M}: \lambda \in \mathbb{B}_{m}\},
\end{aligned}
$$
where $L$ is defined in Theorem \ref{2.6} and $\mathscr{M}=\operatorname{ker}L.$
\end{lem}
\begin{proof}
Since $\mathcal{H}$ is pure,  we have $L L^{*}=I_\mathcal{H},$ hence $L^{*} L=P_{\text {ker } L^{\perp}}.$
By Lemma \ref{lem2.7}, there exists \begin{equation}\label{fuluxuyao}
                                     \Phi: H^{2} \otimes E \rightarrow H^{2} \otimes R
                                    \end{equation}
 such that $I-L^{*} L=\Phi \Phi^{*}.$ This means that $P_{\text {Ker}L}=\Phi \Phi^{*}$. Set $\mathscr{M}=\operatorname{ker} L=\text {ran}\Phi,$ then
by \eqref{mudi2}, \eqref{mudi3}, \eqref{fulu4.9}, Proposition \ref{KH} and \eqref{zuihouyici},
 $$
\begin{aligned}
K_{m}(\mathcal{H})&=\int_{\partial\mathbb{ B}_{m}} \operatorname{trace}( I_{R}-\Phi( z^{(m)}) \Phi( z^{(m)})^{*}) d\sigma_m=\operatorname{rank} \mathcal{H}-\sup \left\{\operatorname{dim} E_{\lambda} \mathscr{M}: \lambda \in \mathbb{B}_{m}\right\} .
\end{aligned}
$$
\end{proof}

By Lemma \ref{lem3.6}, if $\cal H$ is a pure $\omega$-contractive Hilbert module, then $K_{m+1}(\mathcal{H})\leq K_{m}(\mathcal{H}),$ and we define the asymptotic curvature invariant by
 \begin{eqnarray}\label{K(H)}
K (\mathcal{H})=\lim \limits_{m \rightarrow \infty} K_{m}(\mathcal{H}).
\end{eqnarray}

The following lemma establishes the relationship between the ``asymptotic" fiber dimension and the fiber dimension of the submodule.
\begin{lem}\label{Theorem3.6}
For a submodule $\mathcal{M}$ of $H^{2}\otimes \mathbb{C}^N,$ $$\lim\limits_{m\rightarrow\infty}\operatorname{sup} \{\operatorname{dim}~E_{\lambda}\mathcal{M}:\lambda\in \mathbb{B}_m\}=fd(\mathcal{M}).$$
\end{lem}
\begin{proof}
It suffices to prove that $$fd(\mathcal{M})\leq \lim\limits_{m\rightarrow\infty}\operatorname{sup} \{\operatorname{dim}~E_{\lambda}\mathcal{M}:\lambda\in \mathbb{B}_m\}.$$
 For any fixed $\lambda \in \mathbb{B},$ let $\{y_i\}_{i=1}^k$ be an orthonormal basis of $E_{\lambda}\mathcal{M},$ and $\{g_i\}_{i=1}^{\infty}$ be an orthonormal basis of $H^2.$
 By the definition of $E_{\lambda}\mathcal{M},$ there exist $f_i\in \mathcal{M}$ such that $f_i(\lambda)=y_i, 1\leq i\leq k.$ Extend $\{y_i\}_{i=1}^k$ to an orthonormal basis of $\mathbb{C}^N,$ denoted by
$\{y_i\}_{i=1}^N.$
 Write $$f_i=\sum\limits_{j=1}^{N}\sum\limits_{l=1}^{\infty}a_{lij}(g_l\otimes y_j)=\sum\limits_{j=1}^{N}(\sum\limits_{l=1}^{\infty}a_{lij}g_l\otimes y_j),$$
and set $f_{ij}=\sum\limits_{l=1}^{\infty}a_{lij}g_l.$
Since $f_i(\lambda)=y_i,$ we have $f_{ij}(\lambda)=\delta_{ij}.$
Set $$h_{ij}^m=f_{ij}(\lambda_1,\cdots,\lambda_m,0,\cdots),$$
then
\begin{eqnarray}\label{det1}
\lim\limits_{m \rightarrow \infty}\text{det} \left[
                 \begin{array}{cccc}
                   h_{11}^{m} & h_{12}^{m} & \cdots & h_{1k}^{m} \\
                    h_{21}^{m} & h_{22}^{m} & \cdots & h_{2k}^{m} \\
                   \vdots & \vdots & \ddots & \vdots \\
                    h_{k1}^{m} & h_{k2}^{m} & \cdots & h_{kk}^{m} \\
                 \end{array}
               \right]=\text{det} \left[
                 \begin{array}{cccc}
                   f_{11}(\lambda) & f_{12}(\lambda) & \cdots & f_{1k}(\lambda) \\
                    f_{21}(\lambda) & f_{22}(\lambda) & \cdots & f_{2k}(\lambda) \\
                   \vdots & \vdots & \ddots & \vdots \\
                    f_{k1}(\lambda) & f_{k2}(\lambda) & \cdots & f_{kk}(\lambda) \\
                 \end{array}
               \right]=1.
\end{eqnarray}
Hence there exists an $m_0>0$ such that for $m\geq m_0,$
 \begin{eqnarray}\label{det}
\text{det} \left[
                 \begin{array}{cccc}
                   h_{11}^{m} & h_{12}^{m} & \cdots & h_{1k}^{m} \\
                    h_{21}^{m} & h_{22}^{m} & \cdots & h_{2k}^{m} \\
                   \vdots & \vdots & \ddots & \vdots \\
                    h_{k1}^{m} & h_{k2}^{m} & \cdots & h_{kk}^{m} \\
                 \end{array}
               \right]\neq0.
\end{eqnarray}
Notice that $$\mathfrak{g}_i=f_i(\lambda_1,\cdots,\lambda_{m_0},0,\cdots)=\sum\limits_{j=1}^{N}h_{ij}^{m_0}y_j\in E_{(\lambda_1,\cdots,\lambda_{m_0},0,\cdots)}M.$$
We claim that $\mathfrak{g}_1,\cdots,\mathfrak{g}_{k}$ are linearly independent. Indeed,
if $$\mu_1\mathfrak{g}_1+\cdots+\mu_k\mathfrak{g}_k=0,$$ then
$$\mu_1\sum\limits_{j=1}^{N}h_{1j}^{m_0}y_j
+\cdots+\mu_k\sum\limits_{j=1}^{N}h_{kj}^{m_0}y_j=0,$$
thus, $$\sum\limits_{j=1}^{N}\sum\limits_{i=1}^{k}\mu_ih_{ij}^{m_0}y_j=0.$$
This implies that for all $1\leq j\leq N,$
$$\sum\limits_{i=1}^{k}\mu_ih_{ij}^{m_0}=0.$$ By (\ref{det}), we have $\mu_i=0$ for all $1\leq i\leq k.$
\end{proof}

By Lemma \ref{lem3.6} and Lemma \ref{Theorem3.6}, and the similar reasoning in the proof of Lemma \ref{KNf}, we have the following theorem.
\begin{thm}\label{comKH}Let $\mathcal{H}$ be a pure $\omega$-contractive Hilbert module of finite rank, then 
 $$
\begin{aligned}
K(\mathcal{H})&=\operatorname{rank} \mathcal{H}-fd(\mathscr{M}),
\end{aligned}
$$
where $\mathscr{M}=\operatorname{ker}L,$ which is a submodule of $H^2 \otimes R,$ and $L$ is defined in Theorem \ref{2.6}. In particular, if
 $\mathcal{M}\subseteq H^2\otimes \mathbb {C}^N$ is a submodule, then
$$K(\mathcal{M}^\bot)=N-fd(\mathcal{M}).$$
\end{thm}

Now, let $$\phi_m(X)=\sum\limits_{k=1}^m T_k X T_k^*, ~X \in \mathcal{B}(\mathcal{H}).$$
For $m\geq 1,$ set
\begin{equation}
q_m(x)=\frac{(x+1)(x+2)\cdots(x+m)}{m!}.
\end{equation}
We have the following proposition, whose  proof  is inspired by \cite[Throrem C]{Arveson curvature}, and we put it in Appendix.
\begin{prop}\label{2.151}
Let $\mathcal{H}$ be an $\omega$-contractive Hilbert module of finite rank, then
$$
K_m(\mathcal{H})
=\lim _{n \rightarrow \infty} \frac{\operatorname{trace}\left(\sum\limits_{k=0}^n \phi_m^k\left(\Delta_\mathcal{H}^2\right)\right)}{q_m(n)}=\lim _{n \rightarrow \infty} \frac{\operatorname{trace}\left(\phi_m^n\left(\Delta_\mathcal{H}^2\right)\right)}{q_{m-1}(n)}.
$$
\end{prop}

Next, we define the asymptotic Euler characteristic of $\omega$-contractive Hilbert modules of finite rank, and establish its connection to the asymptotic curvature invariant. To do this, we need some preparation.
  Let
$$
\mathbb{M}_\mathcal{H}=\operatorname{span}\left\{f \cdot \Delta_\mathcal{H} \zeta : f \in \mathcal{P}_{\infty}, \zeta \in \mathcal{H}\right\},
$$
and
$$
\mathbb{M}_\mathcal{H}^m=\operatorname{span}\left\{f \cdot \Delta_\mathcal{H} \zeta : f \in \mathcal{P}_m, \zeta \in \mathcal{H}\right\}.
$$
Then $\mathbb{M}_\mathcal{H}(resp. \mathbb{M}_\mathcal{H}^m)$ is a finitely generated $\mathcal{P}_{\infty}(resp.\mathcal{P}_m)$-module.
By \cite[Lemma 10.1]{Rotman} and \cite[Proposition 2.11]{Atiyah},
 \begin{equation}\label{wuqiongyong}
  \chi_{m}(\mathcal{H})=\operatorname{rank} \mathbb{M}_\mathcal{H}^m.
 \end{equation}
Hence by Lemma \ref{lemF}, it is easy to see that $$\operatorname{rank} \mathbb{M}_\mathcal{H}^{m+1}\leq \operatorname{rank} \mathbb{M}_\mathcal{H}^{m}.$$
 Then we have the following lemma.

\begin{lem}\label{thm5.1}
 Let $\mathcal{H}$ be an $\omega$-contractive Hilbert module of finite rank,
then
$$
\chi_{m+1}(\mathcal{H}) \leq \chi_{m}(\mathcal{H}).
$$
\end{lem}
By Lemma \ref{thm5.1},
we define the asymptotic Euler characteristic of $\mathcal{H}$
\begin{equation}\label{cuocuocuo}
\chi(\mathcal{H})=\lim _{m \rightarrow \infty} \chi_{m}(\mathcal{H}) .
\end{equation}
\begin{rem}
It is not difficult to see that the Euler characteristic $\chi(\mathcal{H})$ corresponds to $\operatorname{rank} \mathbb{M}_\mathcal{H}$.
 In fact,
suppose that $\operatorname{dim} \operatorname{ran}\Delta_\mathcal{H}=\mathfrak{n}.$
Let $\left\{e_1, \cdots, e_\mathfrak{n}\right\}$
     be a linear basis of $\operatorname{ran}\Delta_\mathcal{H}.$ Without loss of generality, suppose that
$\left\{e_{1}, \cdots, e_{k}\right\}$
    is a maximal linearly independent set of $\left\{e_1, \cdots, e_\mathfrak{n}\right\}$ in the module $\mathbb{M}_{\mathcal{H}}.$ Obviously, there exists an $m_0> 0$ such that
    for $m>  m_0,$
$\left\{e_{1}, \cdots, e_{k}\right\}$
    is a maximal linearly independent set of $\left\{e_1, \cdots, e_\mathfrak{n}\right\}$ in the module $\mathbb{M}_{\mathcal{H}}^{m}.$
    Hence by Lemma \ref{lemF}, $$\chi(\mathcal{H}) =\lim _{m \rightarrow \infty} \chi_{m}(\mathcal{H})=\lim _{m\rightarrow \infty} \operatorname{rank} \mathbb{M}_\mathcal{H}^m=k =\operatorname{rank} \mathbb{M}_\mathcal{H}.$$
\end{rem}

Similar to \cite[Proposition 4.5]{Arveson curvature} and the corollary of \cite[Proposition 4.9]{Arveson curvature},
$
\left\{m ! \frac{\operatorname{dim}\mathbb{ M}_{\mathcal{H}}^{m,n}}{n^m}\right\}_{n=1}^{\infty}
$
converges and
\begin{equation}\label{zuihou1}
\chi_m(\mathcal{H})=m ! \lim _{n \rightarrow \infty} \frac{\operatorname{dim} \mathbb{M}_{\mathcal{H}}^{m,n}}{n^m},
\end{equation}
where $$\mathbb{M}_{\mathcal{H}}^{m,n}=\operatorname{span}\left\{f \cdot \Delta_\mathcal{H} \zeta : f \in \mathcal{P}_m, \operatorname{deg} f \leq n, \zeta \in \mathcal{H}\right\}.$$
Then we have the following proposition, whose proof will be
 put in Appendix.
\begin{prop}\label{3.21}
Let $\mathcal{H}$ be an $\omega$-contractive Hilbert module of finite rank,
then
$$
\operatorname{dim} \mathbb{M}_{\mathcal{H}}^{m,n}=\operatorname{rank}\left(\sum_{k=0}^n \phi_m^k\left(\Delta_\mathcal{H}^2\right)\right),
$$
and hence $$
\chi_m(\mathcal{H})=m ! \lim _{n \rightarrow \infty} \frac{\operatorname{rank}\left(\sum\limits_{k=0}^n \phi_m^k\left(\Delta_\mathcal{H}^2\right)\right)}{n^m}.
$$
\end{prop}

Recall that $K_m(\mathcal{H})$ is decreasing if $\mathcal{H}$ is pure, and the asymptotic curvature invariant $K (\mathcal{H})=\lim \limits_{m \rightarrow \infty} K_{m}(\mathcal{H}),$
hence we
 have the following inequality.
\begin{prop}\label{KHXH} Let $\mathcal{H}$ be an $\omega$-contractive Hilbert module of finite rank, then
$$K_m(\mathcal{H}) \leq \chi_m(\mathcal{H}),
$$
In particular, if $\mathcal{H}$ is pure, $K(\mathcal{H})\leq\chi(\mathcal{H}).$
\end{prop}
\begin{proof}
Since $$\left(\sum\limits_{k=0}^n \phi_m^k\left(\Delta_\mathcal{H}^2\right)\right)=\left(\sum\limits_{k=0}^n \phi_m^k\left(I- \phi(I)\right)\right)
\leq \left(\sum\limits_{k=0}^n \phi_m^k\left(I- \phi_m(I)\right)\right)=I-\phi_m^{n+1}(I) \leq I,$$ we have
$$\operatorname{trace}\left(\sum_{k=0}^n \phi_m^k\left(\Delta_\mathcal{H}^2\right)\right)\leq \operatorname{rank}\left(\sum_{k=0}^n \phi_m^k\left(\Delta_\mathcal{H}^2\right)\right).$$
Then by Proposition \ref{2.151} and Proposition \ref{3.21}, $$K_m(\mathcal{H}) \leq \chi_m(\mathcal{H}).
$$
Furthermore, if $\mathcal{H}$ is pure, then by \eqref{K(H)} and \eqref{cuocuocuo}, $$K(\mathcal{H})=\lim \limits_{m \rightarrow \infty} K_{m}(\mathcal{H})\leq\lim \limits_{m \rightarrow \infty} \chi_{m}(\mathcal{H}) =\chi(\mathcal{H}).$$
\end{proof}


With the preparations outlined above, we can prove Theorem \ref{thm1.5}, where we generalize Arveson's version of the Gauss-Bonnet-Chern formula to the  infinitely-many-variables case. The proof is based on that of Theorem \ref{thm2.611}. At first, we restate the theorem as the following form.
\begin{thm}
If $\mathcal{M}\subseteq H^2 \otimes \mathbb{C}^N$ is a submodule,
then the
  following statements are equivalent.
\begin{itemize}
     \item[(1)] $K(\mathcal{M}^\bot)=\chi(\mathcal{M}^\bot)$;
     \item[(2)]
  there exists a nonempty open set $U\subseteq \operatorname{mp}(\mathcal{M})$ and polynomials $\{p_i\}_{i=1}^{\mathfrak{m}} \subseteq M$ such that
  $$E_{\lambda}\mathcal{M}=\operatorname{span}\{p_i(\lambda):1\leq i\leq {\mathfrak{m}}\}$$ for all $\lambda \in U;$
     \item[(3)]
     $\mathcal{M}$ is a locally algebraic submodule.
\end{itemize}
\end{thm}
\begin{proof}
(2)$\Rightarrow$(3): Obviously.

(3)$\Rightarrow$(1):
For simplicity, set $l=fd(M).$ By the local algebraicity of $\mathcal{M},$ we can take $\lambda_{0}\in \operatorname{mp}(\mathcal{M})$ and polynomials $\{p_i\}_{i=1}^l \subseteq \mathcal{M}$ such that
 $$E_{\lambda_{0}}\mathcal{M}=\operatorname{span}\{p_i(\lambda_{0}):1\leq i\leq l\}.$$
Write $$p_i=(p_{i,1},\cdots,p_{i,N}).$$
Similar to the proof of Theorem \ref{thm2.611},
 there exist $1\leq i_1<\cdots<i_l\leq N$ such that the determinant
  \begin{eqnarray}\label{eq:41}
\left|
                 \begin{array}{cccc}
                   p_{1,i_1} & p_{1,i_2} & \cdots & p_{1,i_l} \\
                   p_{2,i_1} &p_{2,i_2} & \cdots &p_{2,i_l} \\
                   \vdots & \vdots & \ddots & \vdots \\
                    p_{l,i_1} & p_{l,i_2} & \cdots & p_{l,i_l}\\
                 \end{array}
               \right|\neq 0.
\end{eqnarray}
Since $$\{p_{i,k}:1\leq i\leq l, 1\leq k \leq N\}\subseteq \mathcal{P}_{\infty},$$ there exists an $m_0> 0$ such that $$\{p_{i,k}:1\leq i\leq l, 1\leq k \leq N\}\subseteq \mathcal{P}_{m_0}.$$
By Proposition \ref{KHXH} and Theorem \ref{comKH}, $$K(\mathcal{M}^\bot)\leq \chi(\mathcal{M}^\bot), \quad K(\mathcal{M}^\bot)= N-l.$$
By \eqref{wuqiongyong},
$$\chi_m(\mathcal{M}^\bot)=\text{rank}(\mathbb{M}_{\mathcal{M}^\bot}^m ).$$ Hence in order to prove $K(\mathcal{M}^\bot)=\chi(\mathcal{M}^\bot),$
it suffices to prove that for $m> m_0,$
 \begin{eqnarray}\label{suffices}
 \text{rank}(\mathbb{M}_{\mathcal{M}^\bot}^m) \leq N-l.
\end{eqnarray}
The same reasoning used in the proof of Theorem \ref{thm2.611} demonstrates that for any $1\leq w_1<\cdots<w_{N-l+1}\leq N,$  $\{{P_{\mathcal{M}^\bot}(1\otimes e_{w_i})}\}_{i=1}^{N-l+1}$
  is linearly dependent in the module $\mathbb{M}_{\mathcal{M}^\bot}^m$.
Similar to \eqref{(2.1)}, $\{{P_{\mathcal{M}^\bot}(1\otimes e_{i})}\}_{i=1}^{N}$ generates $\mathbb{M}_{\mathcal{M}^\bot}^m .$ Therefore by Lemma \ref{lemF}, we have
 $$\text{rank}(\mathbb{M}_{\mathcal{M}^\bot}^m )\leq N-l.$$

(1)$\Rightarrow$(2): Suppose that $K(\mathcal{M}^\bot)=\chi(\mathcal{M}^\bot).$ Set $l=fd(\mathcal{M}).$  Hence by Theorem \ref{comKH}, we have
$\chi(\mathcal{M}^\bot)=N-l.$ Then there exists an $m> 0,$ such that $\chi_m(\mathcal{M}^\bot)=N-l.$ By Lemma \ref{lemF}, each maximal linearly independent set of $\{{P_{\mathcal{M}^\bot}(1\otimes e_{i})}\}_{i=1}^{N}$ in the module $\mathbb{M}_{\mathcal{M}^\bot}^m $ has $N-l$ elements. A similar proof to that of Theorem \ref{thm2.611} implies that
 there exists a nonempty open set $U\subseteq \operatorname{mp}(\mathcal{M})$ and polynomials $\{p_i\}_{i=1}^{\mathfrak{m}} \subseteq \mathcal{M}$ such that for all $\lambda \in U,$
  $$E_{\lambda}\mathcal{M}=\operatorname{span}\{p_i(\lambda):1\leq i\leq {\mathfrak{m}}\}.$$
\end{proof}

Obviously, the asymptotic curvature invariant and the asymptotic Euler characteristic are both unitarily invariant. We list this as the following proposition.

\begin{prop}\label{thm5.2}
If $H$ and $H^{\prime}$ are unitarily equivalent; that is, there exists a unitary operator $U: H \rightarrow H ^{\prime}$, such that for all $f\in \mathcal{P}_{\infty},$ and all $\zeta\in H,$
$$
U(f \cdot \zeta)=f \cdot U(\zeta).
$$
Then $\chi(H)=\chi(H^{\prime})$. Moreover, if $H$ and $H^{\prime}$ are pure, then $K(H)=K(H^{\prime})$.
\end{prop}

\section{An {\color{red}example} of homogeneous submodule containing no nonzero polynomials}
 It is straightforward to verify that any homogeneous submodule of $H_d^2$
  is polynomially generated. Consequently, Arveson's version of the Gauss-Bonnet-Chern formula holds for homogeneous quotient modules in $H_d^2.$
In this section, we give an example of homogeneous submodule $\mathcal{M}$ of $H^2,$ which contains no nonzero polynomials.
Then by Theorem \ref{thm1.5} and Proposition \ref{KHXH}, $$K(\mathcal{M}^{\bot})<\chi(\mathcal{M}^{\bot}).$$

Let $$f=\sum\limits_{i=1}^{\infty} \frac{1}{(i+1)!} z_i,$$
To simplify the notation,  write $\lambda_i=\frac{1}{(i+1)!}.$
Consider the homogeneous submodule $\mathcal{M}_f$ in $H^2$ generated by $f.$ Obviously, $\mathcal{M}_f$ is homogeneous:
$$
\mathcal{M}_f=\mathop{\oplus}_{n=1}^{\infty} \mathcal{M}_{n},
$$
where
$$
 \mathcal{M}_{n}=\overline{\operatorname{span}\left\{z^{\alpha}f: |\alpha|=n-1\right\}}, \quad n \geq 1 .\\
$$

\begin{lem}\label{lem5.6}
$\mathcal{M}_{n}$ contains no nonzero ploynomials.
\end{lem}
\begin{proof}
 Let $H_n$ be the space of homogeneous functions in $H^2$ of degree $n.$  Given a ploynomial $q\in H_n \cap \mathcal{P}_{\infty},$ then there exists a $d_0> n$ such that $$q\in \mathcal{P}_{d_0}\cap H_n.$$ We claim that for all $p\in \mathcal{P}_\infty\cap H_{n-1},$ $$||q-p \cdot f||^2\geq  \frac{||q||^2}{4d_0^3(d_0+1)! (d_0+4)!^2}.$$
Indeed, let
 $P_{d_0}$ be the orthogonal projection from $H^2$ onto $H_{d_0}^2.$
Write
$$p=(I-P_{d_0})p +P_{d_0}p.$$
For simplicity, set $$ p_1=(I-P_{d_0})\,p,\quad p_2=P_{d_0}\,p,$$ and
$$ f_1=\sum \limits _{i=1}^{d_0}\lambda_i z_i, \quad f_2=\sum \limits _{i=d_0+1}^{\infty}\lambda_i z_i . $$
Since for $|\alpha|=n-2, ~i,j> d_0,~l>d_0~ \text{and} ~z^\beta \in \mathcal{P}_{d_0},$ $\langle z_iz^\beta,z_j z_lz^\alpha\rangle=0;$ therefore,
 $$\langle  p_2 f_2,p_1 f_2\rangle=\left\langle  p_2 \sum \limits _{i=d_0+1}^{\infty}\lambda_i z_i,p_1 \sum \limits _{i=d_0+1}^{\infty}\lambda_i z_i\right\rangle
 =\sum \limits _{i=d_0+1}^{\infty}  \sum \limits _{j=d_0+1}^{\infty}\lambda_i \overline{\lambda_j} \langle  p_2 z_i,p_1 z_j\rangle=0.$$
By direct calculations, $$\langle q,p_1 f_2\rangle=\langle  p_2 f_1,p_1 f_2\rangle=\langle  p_1 f_1,p_1 f_2\rangle=\langle q,  p_2 f_2\rangle
 =\langle  p_2 f_1,  p_2 f_2\rangle=\langle q,p_1 f_1\rangle=\langle p_2 f_1,p_1 f_1\rangle=0.$$
It follows that
$$\langle q- p_2 f_1,p_1 f_2\rangle=\langle  p_2 f_2+p_1 f_1,p_1 f_2\rangle=\langle q- p_2 f_1,  p_2 f_2+p_1 f_1\rangle=0.$$
Then
$$||q-p  f||^2=||q- p_2 f_1- p_2 f_2-p_1 f_1-p_1f_2||^2 =||q- p_2 f_1||^2+|| p_2 f_2+p_1 f_1||^2+||p_1 f_2||^2.$$
Next, we will estimate $dist(q,\mathcal{M}_{n})$ in two cases.\\
 $Case ~1.$ $|| p_2|| \leq\frac{1}{2}\frac{||q||}{\sum \limits_{i=1}^{d_0} \lambda_i}.$

 Notice that
 $\left\| p_2 f_1\right\| \leq \sum \limits_{i=1}^{d_0} \lambda_i\| p_2\|,$ then
$$||q-p  f||\geq \left\|q- p_2 f_1\right\| \geq\|q\|-\left\| p_2 f_1\right\|\geq \|q\|-\sum \limits_{i=1}^{d_0} \lambda_i\| p_2\|\geq \frac{\|q\|}{2}.$$
$Case ~2.$ {\color{red}$|| p_2|| > \frac{1}{2}\frac{||q||}{\sum \limits_{i=1}^{d_0} \lambda_i}.$}

  For $|\gamma|=n-1$ and $z^\gamma \in \mathcal{P}_{d_0},$ $\left\| z^\gamma z_i\right\|^2\geq\frac{\| z^\gamma\|^2}{n}, i\geq d_0 +1.$
Then $$|| p_2 f_2||^2=\left\|\sum \limits_{i=d_0+ 1}^{\infty}  p_2 \lambda_i z_i\right\|^2=\sum \limits_{i=d_0 +1}^{\infty} \lambda_i\left\| p_2 z_i\right\|^2\geq\sum \limits_{i=d_0+1}^{\infty} \lambda_i\frac{\| p_2\|^2}{n}\geq\frac{\|q\|^2}{\left(\sum \limits_{i=1}^{d_0} \lambda_i\right)^2}\sum \limits_{i=d_0+1}^{\infty}\frac{ \lambda_i}{4n}.$$
Set $$h^2=\frac{\|q\|^2}{\left(\sum \limits_{i=1}^{d_0} \lambda_i\right)^2}\sum \limits_{i=d_0+1}^{\infty}\frac{ \lambda_i}{4n}.$$
 $Case ~2.1.$ $||p_1||\leq\frac{h}{2||M_{f_1}||}.$  $$||q-p  f||\geq||p_2 f_2+p_1 f_1||\geq(||p_2 f_2||-||p_1 f_1||)\geq\frac{h}{2}.$$
 $Case ~2.2.$ $||p_1|| > \frac{h}{2||M_{f_1}||}.$

 Notice that $$\left\| z^{\gamma^{\prime}} z_{d_0+1}\right\|^2\geq\frac{\| z^{\gamma^{\prime}}\|^2}{n}~ \text{for}~ |\gamma^{\prime}|=n-1, z^{\gamma^{\prime}} \perp \mathcal{P}_{d_0},$$
$$\frac{\lambda_{d_0+1}}{d_0}-\sum_{i=d_0+2}^{\infty} \lambda_i=
\frac{1}{d_0(d_0+1)!}-\sum_{i=d_0+2}^{\infty} \frac{1}{(i+1)!}\geq\frac{2}{(d_0+3)!}>0,$$
and
$$||M_{f_1}||\leq\sum_{i=1}^{d_0} \lambda_i\leq d_0+4 ,$$
  so
$$
\begin{aligned}
||q-p  f||& \geq \left\|p_1 f_2\right\| \\&=\left\|\sum_{i=d_0+1}^{\infty} \lambda_iz_i p_1\right\|
 \\&\geqslant\left\|\lambda_{d_0+1} z_{d_0+1} p_1\right\|-\sum_{i=d_0+2}^{\infty} ||\lambda_i z_i p_1 \| \\
& \geq\lambda_{d_0+1}\frac{\left\|p_1\right\|}{\sqrt{n}}-\sum_{i=d_0+2}^{\infty} \lambda_i\left\|p_1\right\|
\\
& \geq\lambda_{d_0+1}\frac{\left\|p_1\right\|}{d_0}-\sum_{i=d_0+2}^{\infty} \lambda_i\left\|p_1\right\|\\&
=\left\|p_1\right\|\left(\frac{\lambda_{d_0+1}}{d_0}-\sum_{i=d_0+2}^{\infty} \lambda_i\right)\\
&  \geq \frac{h}{2||M_{f_1}||}\left(\frac{\lambda_{d_0+1}}{d_0}-\sum_{i=d_0+2}^{\infty} \lambda_i\right)\\
&  \geq \frac{h}{(d_0+4)!}.
\end{aligned}
$$
We conclude $$||q-p  f||^2\geq min \left\{\frac{\|q\|^2}{4}, \frac{h^2}{4}, \frac{h^2}{(d_0+4)!^2}\right\}= min \left\{\frac{\|q\|^2}{4}, \frac{h^2}{(d_0+4)!^2}\right\}.$$
By the definition of $h^2,$ it is easy to see that $$h^2\geq \frac{\|q\|^2}{4d_0^3(d_0+1)!}.$$
Then $$||q-p  f||^2 \geq \frac{||q||^2}{4d_0^3(d_0+1)! (d_0+4)!^2}.$$
Therefore, $\mathcal{M}_{n}$ contains no nonzero polynomials.

\end{proof}
\begin{prop}
The homogeneous submodule $\mathcal{M}_f$ contains no nonzero ploynomials.
\end{prop}
\begin{proof}
Suppose that there is a ploynomial $p$ such that $p\in\mathcal{M}_f,$ then there is an $n \in \mathbb{Z}^+$ such that
$$
p=\sum_{i=1}^n g_i, \quad g_i \in \mathcal{M}_i,
$$
It follows that for all $1\leq k\leq n,$ $
g_k=E_k p
$
are ploynomials,
where $E_k$ denotes the orthogonal projection from $H^2$ onto $H_k.$
By Lemma \ref{lem5.6}, $g_k=0.$ This implies that $p=0.$
\end{proof}

\section{The finite defect problem of submodules in $H^2$}

The present section is devoted to the finite defect problem of submodule in $H^2.$ The main result is the following theorem.
\begin{thm}\label{thm6.3}
If $\mathcal{M}$ is a nonzero submodule of $H^2$ of finite rank, then $\mathcal{M}=H^2.$
\end{thm}
To prove the theorem, we need the following lemmas. The proof of Lemma \ref{lem6.1} can be obtained through direct computation; therefore, we omit this proof.
\begin{lem}\label{lem6.1} Let $\mathcal{M}$ be a nonzero submodule of $H^2,$ then $$\Delta_\mathcal{M}^2 =
 \left(P_\mathcal{M}E_0P_\mathcal{M} + \sum \limits_{i=1}^{\infty} [ M_i, P_{\mathcal{M}^\bot}][P_{\mathcal{M}^\bot}, M_i^*]\right) \bigg|_\mathcal{M} .$$
\end{lem}

To simplify the notation,
 for every $f \in H^2\setminus \{0\}$, set
$$
N(f)=min \left\{n \in \mathbb{N} : E_n f \neq 0\right\},
$$
where $E_n$ denotes the orthogonal projection from $H^2$ onto $H_n.$ It is easy to see that for nonzero $f,g \in H^2,$
$$N(fg)=N(f)+N(g) .$$
A polynomial $f$ is said to divide another polynomial $g$, denoted $f \mid g$, if there exists a polynomial $h$ such that $g = fh;$ and the notation $f \nmid g$ indicates that $f$ does not divide $g$.

\begin{lem}\label{prop6.5}
Let $\mathcal{M}$ be a nonzero submodule of $H^2$ of finite rank, then  $P_\mathcal{M} 1\neq0.$
\end{lem}
\begin{proof}
Supposed that $P_\mathcal{M} 1=0$, then there is a nonzero $f\in M$ and $k \in \mathbb{N}^{+}$ such that
 \begin{equation}\label{Nf}
       N(f)=\min \{N(g): g \in \mathcal{M} \setminus \{0\}  \}=r \geq 1 \quad  \text{and}\quad
N\left(M_k^* E_r f\right)=N(f)-1 \geq 0.
    \end{equation}
By Lemma \ref{lem6.1}, for $h \in \mathcal{M},$ $\Delta_{\mathcal{M}} h=0$ if and only if $E_0 h=0$ and $P_{\mathcal{M}^\bot} M_i^* h=0, \,i\geq 1.$
Hence from the fact that $\mathcal{M}$ has finite rank, there exists an $m>0$ such that
 $$\operatorname{dim} \operatorname{span}\{P_{M^\bot} M_{k}^* z_{i}f, i=1,2,\cdots\}=m-1<\infty.$$
 Then for every $t \in \mathbb{N}$, there are $\widetilde{\lambda}_{1}^{(t)}, \cdots, \widetilde{\lambda}_{m}^{(t)},$ which are not all zero, such that
$$
P_{\mathcal{M}^{\perp}}\left(\sum_{i=1}^m \widetilde{\lambda}_{i}^{(t)} M_k^* M_{m t+i+k} f\right)=\sum_{i=1}^m  \widetilde{\lambda}_{i}^{(t)} P_{\mathcal{M}^{\perp}} M_k^* M_{m t+i+k} f=0.
$$
Set $g_t=\sum\limits_{i=1}^m  \widetilde{\lambda}_{i}^{(t)} M_{m t+i+k} f$, we have both
$g_t \in \mathcal{M}$ and $M_k^* g_t \in \mathcal{M}.$
It follows from $\text{rank}\, \mathcal{M}< \infty$ that $\{\Delta_\mathcal{M} M_k^* g_t\}_{t=1}^{\infty}$ is linearly dependent.
 Hence there is a positive integer $n$ and $\lambda_1, \lambda_2, \cdots, \lambda_n,$ which are not all zero, such that
$$
\Delta_\mathcal{M} \sum_{t=1}^n \lambda_t M_k^* g_t=\sum_{t=1}^n \lambda_t \Delta_\mathcal{M} M_k^* g_t=0.
$$
By Lemma \ref{lem6.1}, it is easy to see that
 \begin{equation}\label{touying}
       P_{\mathcal{M}^{\perp}} M_j^* \sum_{t=1}^n \lambda_t M_k^* g_t=0, ~\text{for all} ~j \in \mathbb{N}.
    \end{equation}
Notice that $\lambda_1,\cdots,\lambda_n$ are not all zero, then there exists a $ a \in \{1,\cdots,n\}$ and $ i_a \in \{1,\cdots,m\}$ such that $\lambda_a\widetilde{\lambda}_{i_a}^{(a)}\neq 0.$
By \eqref{touying}, $M_{m a+i_a+k}^* \sum\limits_{t=1}^n \lambda_t M_k^* g_t \in \mathcal{M}.$
Then by \eqref{Nf},
\begin{eqnarray}\label{eq104}
E_{r-1}\left(M_{m a+i_a+k}^* \sum\limits_{t=1}^n \lambda_t M_k^* g_t\right)= 0.
\end{eqnarray}
Since $E_{i} M_j^*=M_j^* E_{i+1},$ we have
\begin{align}\label{eq100}
E_{r-1} \left(M_{m a+i_a+k}^* \sum\limits_{t=1}^n \lambda_t M_k^* g_t\right)  &= M_{m a+i_a+k}^* \sum\limits_{t=1}^n \lambda_t M_k^* \left(\sum\limits_{i=1}^m  \widetilde{\lambda}_{i}^{(t)} M_{m t+i+k} E_r f\right) \nonumber \\
&= M_{m a+i_a+k}^* M_k^* \sum\limits_{t=1}^n \left(\sum\limits_{i=1}^m \lambda_t \widetilde{\lambda}_{i}^{(t)} M_{m t+i+k} E_r f\right).
\end{align}
Write $E_rf=f_3+f_4,$ where $f_3 \in \operatorname{ker} M_k^*,$ $f_4 \in \overline{\operatorname{ran} M_k}.$
 For monomials $z^\alpha \in \mathcal{P}_{\infty},~|\alpha|=r,$
$$M_k^* M_{m t+i+k} (z^\alpha) = \frac{r}{r+1}M_{m t+i+k} M_k^* (z^\alpha), \quad 1\leq t\leq n, 1\leq i\leq m ,$$
then
 \begin{equation}\label{eq101}
\begin{aligned}
M_{m a+i_a+k}^* &M_k^* \sum_{t=1}^n \sum_{i=1}^m \lambda_t \widetilde{\lambda}_{i}^{(t)} M_{m t+i+k} E_r f \nonumber\\&=M_{m a+i_a+k}^* M_k^* \sum_{t=1}^n \sum_{i=1}^m \lambda_t \widetilde{\lambda}_{i}^{(t)} M_{m t+i+k} (f_3)+M_{m a+i_a+k}^* M_k^* \sum_{t=1}^n \sum_{i=1}^m \lambda_t \widetilde{\lambda}_{i}^{(t)} M_{m t+i+k} (f_4)  \nonumber\\&=M_{m a+i_a+k}^* M_k^* \sum_{t=1}^n \sum_{i=1}^m \lambda_t \widetilde{\lambda}_{i}^{(t)} M_{m t+i+k} (f_4) \\&=M_{m a+i_a+k}^* \sum_{t=1}^n \sum_{i=1}^m \lambda_t \widetilde{\lambda}_{i}^{(t)}M_k^* M_{m t+i+k} (f_4) \nonumber\\&= \frac{r}{r+1} M_{m a+i_a+k}^* \sum_{t=1}^n \sum_{i=1}^m \lambda_t \widetilde{\lambda}_{i}^{(t)} M_{m t+i+k} M_k^*(f_4)  .
\end{aligned}
\end{equation}
Therefore by \eqref{eq104} and \eqref{eq100},
 \begin{equation}\label{eq1011}
M_{m a+i_a+k}^* \sum_{t=1}^n \sum_{i=1}^m \lambda_t \widetilde{\lambda}_{i}^{(t)} M_{m t+i+k} M_k^*(f_4)=0.
\end{equation}
On the other hand, by \eqref{Nf},
\begin{equation}\label{f4}
\begin{aligned}
M_k^*(f_4)\neq0,
\end{aligned}
\end{equation}
and write \begin{equation}\label{eq.10111}
M_k^*(f_4)=\sum\limits_{|\beta|=r-1,~b_{\beta}\neq0}b_{\beta} z^{\beta}.
\end{equation}
Let
$$J=\{z^{\beta}:|\beta|=r-1,~ b_{\beta}\neq0\},$$and
$$J_{j}=\{z^{\beta} \in J: z_{m a+i_a+k}^j |z^{\beta} ~\text{and}~z_{m a+i_a+k}^{j+1} \nmid z^{\beta}\}, \quad 0\leq j\leq r-1.$$
Obviously there exists at least one $ j \in \{0,\cdots, r-1\}$ such that $J_j \neq \emptyset.$
Set $$\mathfrak{l}=max\{j:0\leq j\leq r-1, J_{j}\neq \emptyset \}.$$ By \eqref{eq.10111},
$$\sum_{t=1}^n \sum_{i=1}^m \lambda_t \widetilde{\lambda}_{i}^{(t)} M_{m t+i+k} M_k^*(f_4)=\sum_{t=1}^n \sum_{i=1}^m \lambda_t \widetilde{\lambda}_{i}^{(t)} M_{m t+i+k} \sum\limits_{|\beta|=r-1,~b_{\beta}\neq0}b_{\beta} z^{\beta}. $$
For $z^{\beta}\in J_{\mathfrak{l}},$ it is not difficult to check that
 $$\left\langle \lambda_a \widetilde{\lambda}_{i_a}^{(a)} b_{\beta} z_{m a+i_a+k}z^{\beta}, \sum_{t=1}^n \sum_{i=1}^m \lambda_t \widetilde{\lambda}_{i}^{(t)} M_{m t+i+k} M_k^*(f_4)-\lambda_a \widetilde{\lambda}_{i_a}^{(a)} b_{\beta} z_{m a+i_a+k}z^{\beta}\right\rangle =0.$$
 Hence   \begin{equation}\label{eq10111}
M_{m a+i_a+k}^* \sum_{t=1}^n \sum_{i=1}^m \lambda_t \widetilde{\lambda}_{i}^{(t)} M_{m t+i+k} M_k^*(f_4)\neq 0,
\end{equation}
which contradicts to \eqref{eq1011}.

\end{proof}
\begin{lem}\label{108}
Given $g \in H_n, n>0,$ if there exists a $k \in \mathbb{N}^{+}$ and a set $\left\{\lambda_i\in \mathbb{C}:i\geq 1\right\},$ such that
\begin{itemize}
     \item[(1)] $\text{for all}~ j \geq 0,~\left\|\sum\limits_{i=1}^k \lambda_{j k+i} M_{j k+i} g\right\|^2=\sum\limits_{i=1}^k|\lambda_{j k+i}|^2\|g\|^2$;
     \item[(2)]
$\text{for all}~  j \geq 0,~\sum\limits_{i=1}^k\left|\lambda_{k j+i}\right|^2>0,$
\end{itemize}
then $g=0.$
\end{lem}
\begin{proof} Write
$$
g=\sum_{|\alpha|=n,~ a_\alpha \neq 0} a_\alpha z^\alpha.
$$
 Set
$$A=\{z^\alpha:|\alpha|=n,~ a_\alpha \neq 0\},\quad A_1=\{z^\alpha \in A : z_t \nmid z^\alpha~ \text{for all} ~1 \leq t \leq k \},$$ and
$$
\begin{gathered}
A_m=\left\{z^\alpha \in A \backslash \bigcup_{s=1}^{m-1} A_s : z_t \nmid  z^\alpha~ \text{for all} ~ k(m-1)+1 \leq t \leq k m\right\}, ~\text{if} ~2\leq m \leq n. \\
\end{gathered}
$$
Let \begin{eqnarray}\label{104}
A_{n+1}=A\setminus \mathop{\cup}\limits_{m=1}^{n}A_m.
\end{eqnarray}
We claim that $$A_{n+1}\subseteq \{z^\alpha \in A: z_t \nmid z^\alpha~ \text{for all}~ kn+1 \leq t \leq k(n+1)\}.$$
In fact, for any $z^\alpha \in A_{n+1},$ by \eqref{104}, we have $z^\alpha \notin A_m$ for all $1\leq m\leq n.$ By the definition of $A_m,$ there exist
 $k(t-1)+1\leq i_t\leq kt, ~1 \leq t\leq n, $
 such that $z_{i_t} |z^\alpha$ for all $1 \leq t\leq n.$
 Notice that $\text{deg} z^\alpha=n,$ hence $z_t\nmid z^\alpha$ for all $kn+1 \leq t \leq k(n+1).$

Obviously, $A=\sqcup_{m=1}^{n+1} A_m,$ where $\sqcup$ means the disjoint union.
Write
$$
g=g_1+\cdots+g_{n+1},\quad
\text{where}~
g_m=\sum_{ z^\alpha \in A_m} a_\alpha z^\alpha,~  1 \leq m \leq n+1,
$$
then $\{g_m\}_{m=1}^{n+1}$ is pairwise orthogonal.
Hence for every $1 \leq j \leq n+1,$
\begin{eqnarray}\label{6.8}
\begin{aligned}
 \left\|\sum_{i=1}^k \lambda_{(j-1) k+i} M_{(j-1) k+i} g\right\|^2&=\sum_{i=1}^k|\lambda_{(j-1) k+i}|^2\|g\|^2\\&=\sum_{i=1}^k\left|\lambda_{(j-1) k+i}\right|^2\left\|g_j\right\|^2+\sum_{i=1}^k|\lambda_{(j-1) k+i}|^2\left\|\sum_{m \neq j} g_m\right\|^2.
 \end{aligned}
\end{eqnarray}
Notice that for all $(j-1) k+1 \leq w_1, w_2 \leq j k, ~t \neq j ,$
$$
\begin{aligned}
\langle M_{w_1} g_j, M_{w_2} g_t\rangle&=\langle M_{w_2}^* M_{w_1} g_j, g_t\rangle\\&=\langle M_{w_2}^* M_{w_1} \sum_{ z^\alpha \in A_j} a_\alpha z^\alpha, g_t\rangle\\&= \sum_{ z^\alpha \in A_j} a_\alpha \langle M_{w_2}^* M_{w_1} z^\alpha, g_t\rangle
\\&= \sum_{ z^\alpha \in A_j} a_\alpha\frac{1}{n+1} \delta_{w_1 w_2}\langle z^\alpha, g_t\rangle\\&=\frac{1}{n+1} \delta_{w_1 w_2}\langle g_j, g_t\rangle=0,
\end{aligned}
$$
where $\delta_{w_1 w_2}$ denotes the Kronecker delta.
Moreover for all $(j-1) k+1 \leq r_1\neq r_2 \leq j k,$
$$
\begin{aligned}
\langle M_{r_1} g_j, M_{r_2} g_j\rangle=\langle M_{r_2}^* M_{r_1} g_j, g_j\rangle=\left\langle M_{r_2}^* M_{r_1} \sum_{ z^\alpha \in A_j} a_\alpha z^\alpha, g_j\right\rangle= \sum_{ z^\alpha \in A_j} a_\alpha \langle M_{r_2}^* M_{r_1} z^\alpha, g_j\rangle=0.
\end{aligned}
$$
Theorefore
\begin{equation}\label{111}
\begin{aligned}
&\left\|\sum_{i=1}^k \lambda_{(j-1) k+i} M_{(j-1) k+i} g\right\|^2 \\& =\left\|\sum_{i=1}^k \lambda_{(j-1) k+i} M_{(j-1) k+i} g_j\right\|^2+\left\|\sum_{i=1}^k \lambda_{(j-1) k+i} M_{(j-1) k+i} \sum_{m \neq j} g_m\right\|^2 \\
& =\sum_{i=1}^k\left|\lambda_{(j-1) k+i}\right|^2\left\|M_{(j-1) k+i} g_j\right\|^2+\left\|\sum_{i=1}^k \lambda_{(j-1) k+i} M_{(j-1) k+i} \sum_{m \neq j} g_m\right\|^2 \\
& =\sum_{i=1}^k\left|\lambda_{(j-1) k+i}\right|^2\left\|\sum_{ z^\alpha \in A_j} a_\alpha M_{(j-1) k+i} z^\alpha\right\|^2+\left\|\sum_{i=1}^k \lambda_{(j-1) k+i} M_{(j-1) k+i} \sum_{m \neq j} g_m\right\|^2 \\
& =\sum_{i=1}^k\left|\lambda_{(j-1) k+i}\right|^2\sum_{ z^\alpha \in A_j}| a_\alpha|^2\left\|M_{(j-1) k+i} z^\alpha\right\|^2+\left\|\sum_{i=1}^k \lambda_{(j-1) k+i} M_{(j-1) k+i} \sum_{m \neq j} g_m\right\|^2 \\
& =\sum_{i=1}^k\left|\lambda_{(j-1) k+i}\right|^2 \sum_{ z^\alpha \in A_j}|a_\alpha|^2 \frac{1}{n+1}\left\|z^\alpha\right\|^2+\left\|\sum_{i=1}^k \lambda_{(j-1) k+i} M_{(j-1) k+i} \sum_{m \neq j} g_m\right\|^2 \\
& =\frac{1}{n+1} \sum_{i=1}^k\left|\lambda_{(j-1) k+i}\right|^2\left\|g_j\right\|^2+\left\|\sum_{i=1}^k \lambda_{(j-1) k+i} M_{(j-1) k+i} \sum_{m \neq j}g_m\right\|^2.
\end{aligned}
\end{equation}
Since $\left(M_1, M_2, \cdots\right)$ is an $\omega$-contraction, by \eqref{3.210},
\begin{equation}\label{112}
\left\|\sum_{i=1}^k \lambda_{(j-1) k+i} M_{(j-1) k+i} \sum_{m \neq j}^{k+1} g_m\right\|^2 \leq \sum_{i=1}^k\left|\lambda_{(j-1) k+i}\right|^2\left\|\sum_{m \neq j}^{k+1} g_m\right\|^2.
\end{equation}
By \eqref{6.8}, \eqref{111} and \eqref{112},
$$
\frac{1}{n+1} \sum_{i=1}^k\left|\lambda_{(j-1) k+i}\right|^2\left\|g_j\right\|^2 \geq \sum_{i=1}^k\left|\lambda_{(j-1) k+i}\right|^2\left\|g_j\right\|^2.
$$
Therefore
$$
\frac{1}{n+1}\left\|g_j\right\|^2 \geq\left\|g_j\right\|^2.
$$
This implies that
$
\left\|g_j\right\|^2=0,
$
hence $g=0.$
\end{proof}

\noindent$The ~proof ~of ~Theorem~ \ref{thm6.3}.$ For $i\geq1,$
$$
\Delta_\mathcal{M}^2 (P_\mathcal{M}z_i)=\left(P_\mathcal{M}-\sum \limits_{j=1}^{\infty} M_j P_{\mathcal{M}} M_j^*\right)P_Mz_i=P_\mathcal{M}z_i-\sum \limits_{j=1}^{\infty} M_j P_{\mathcal{M}} M_j^*z_i=P_\mathcal{M}z_i-z_iP_\mathcal{M} 1.
$$
Suppose that $\mathcal{M}$ has finite rank, and set $$n=\text{dim span}\{\Delta_\mathcal{M}^2 (P_\mathcal{M}z_i)|i\geq1\}+1.$$
It follows that there exist $\lambda_1,\cdots,\lambda_n,$ which are not all zero such that
$$\sum_{i=1}^n \lambda_i(P_\mathcal{M}z_i-z_iP_\mathcal{M}1)=\sum\limits_{i=1}^n \lambda_i \Delta_\mathcal{M}^2 P_\mathcal{M} z_i=0.$$
Hence $$\sum_{i=1}^n \lambda_i z_iP_\mathcal{M}1=\sum_{i=1}^n \lambda_iP_\mathcal{M}z_i,$$ which implies that
\begin{align}\label{zhangruoyu1}
\left\|\sum_{i=1}^n \lambda_i z_i P_\mathcal{M} 1\right\|^2 & =\left\langle\sum_{i=1}^n \lambda_i P_\mathcal{M} z_i, \sum_{i=1}^n \lambda_i z_i P_\mathcal{M} 1\right\rangle \nonumber\\
& =\sum\limits_{1\leq i \neq j\leq n}\left\langle\lambda_i P_\mathcal{M} z_i, \lambda_j z_j P_\mathcal{M} 1\right\rangle  +\sum_{i=1}^n\left\langle\lambda_i P_\mathcal{M} z_i, \lambda_i z_i P_\mathcal{M} 1\right\rangle
\nonumber\\&=
\sum\limits_{1\leq i \neq j\leq n}\lambda_i\overline{\lambda_j} \left\langle M_j^*(I-P_{\mathcal{M}^\perp}) z_i,   P_\mathcal{M} 1\right\rangle  +\sum_{i=1}^n |\lambda_i|^2 \left\langle P_\mathcal{M} M_i^*P_\mathcal{M} z_i, P_\mathcal{M} 1\right\rangle \nonumber\\&=
\sum\limits_{1\leq i \neq j\leq n}\lambda_i\overline{\lambda_j} \left\langle M_j^* z_i,   P_\mathcal{M} 1\right\rangle  +\sum_{i=1}^n |\lambda_i|^2 \left\langle P_\mathcal{M} M_i^*P_\mathcal{M} z_i, P_\mathcal{M} 1\right\rangle
\nonumber\\&=
0 +\sum_{i=1}^n |\lambda_i|^2 \left\langle P_\mathcal{M}P_\mathcal{M}  M_i^*z_i, P_\mathcal{M} 1\right\rangle
\nonumber\\&=\sum_{i=1}^n|\lambda_i|^2\langle 1, P_\mathcal{M} 1\rangle.
\end{align}
Write \begin{equation}\label{PM1}
      P_\mathcal{M} 1=c+g,
      \end{equation}
       where $c$ is a constant and $g(0)=0.$
Notice that
\begin{equation}\label{zhangruoyu2}
     \left\|\sum_{i=1}^n \lambda_i z_i P_\mathcal{M }1\right\|^2  =\left\|\sum_{i=1}^n \lambda_ic z_i \right\|^2 +\left\|\sum_{i=1}^n \lambda_i z_i g\right\|^2=
\sum_{i=1}^n |\lambda_i|^2|c|^2 +\left\|\sum_{i=1}^n \lambda_i z_i g\right\|^2,
      \end{equation}
and
\begin{equation}\label{zhangruoyu3}
    \sum_{i=1}^n|\lambda_i|^2\langle 1, P_\mathcal{M} 1\rangle=\sum_{i=1}^n|\lambda_i|^2(|c|^2+||g||^2),
      \end{equation}
therefore by \eqref{zhangruoyu1}, \eqref{zhangruoyu2} and \eqref{zhangruoyu3},
 $\left\|\sum\limits_{i=1}^n \lambda_i z_i g\right\|^2 =\sum\limits_{i=1}^n|\lambda_i|^2||g||^2.$
Set $g=\sum\limits_{m=1}^{\infty}g_m,$ where $g_m\in H_m.$ It is easy to see that
$$\left \langle \sum\limits_{i=1}^n \lambda_i z_i g_m, \sum\limits_{i=1}^n \lambda_i z_i g_k\right \rangle=0, \quad m\neq k. $$
Thus,
\begin{equation}\label{132}
\sum\limits_{m=1}^{\infty}\left\|\sum\limits_{i=1}^n \lambda_i z_i g_m\right\|^2=\left\|\sum\limits_{i=1}^n \lambda_i z_i g\right\|^2 =\sum\limits_{i=1}^n|\lambda_i|^2||g||^2=\sum\limits_{m=1}^{\infty}\sum\limits_{i=1}^n|\lambda_i|^2||g_m||^2.
\end{equation} Since $\left(M_1, M_2, \ldots\right)$ is an $\omega$-contraction, by \eqref{3.210},
\begin{equation}\label{142}
\left\|\sum\limits_{i=1}^n \lambda_i z_i g_m\right\|^2\leq \sum\limits_{i=1}^n|\lambda_i|^2||g_m||^2.
\end{equation}
From \eqref{132}, \eqref{142},
it is easy to see that for all $m\geq 1,$ $$\left\|\sum\limits_{i=1}^n \lambda_i z_i g_m\right\|^2=\sum\limits_{i=1}^n|\lambda_i|^2||g_m||^2.$$
The same reasoning implies that there exist $\{\lambda_{jn+i}\in \mathbb{C}:j \geq 0, 1\leq i \leq n\}$ such that
  $$\left\|\sum\limits_{i=1}^n \lambda_{j n+i} M_{j n+i} g\right\|^2=\sum\limits_{i=1}^n|\lambda_{j n+i}|^2\|g_m\|^2,~\text{for}~ m \geq 1~ \text{and} ~j\geq 0,$$
and
$$\sum\limits_{i=1}^n|\lambda_{jn+i}|^2>0.$$
By Lemma \ref{108}, for all $m\geq 1,$
 $g_m=0,$ and hence $g=0.$ By Proposition \ref{prop6.5}, $P_\mathcal{M}1 \neq 0.$ Then by \eqref{PM1}, $P_\mathcal{M}1=1,$ $i.e.~ 1\in M.$
  Therefore $\mathcal{M}=H^2.$
\hfill \qedsymbol

\section{Appendix}

In this section, we will prove \eqref{oula}, Proposition \ref{2.151} and Proposition \ref{3.21}.

Let \( \mathcal{H} \) be a \( d \)-contractive Hilbert module of finite rank.
Suppose that $\operatorname{dim} \operatorname{ran}\Delta_\mathcal{H}=n.$
Let \begin{equation}\label{bianhao7}
 A=\left\{e_1, \cdots, e_n\right\}
    \end{equation}
     be a linear basis of $\operatorname{ran}\Delta_\mathcal{H}.$ Without loss of generality, suppose that
      \begin{equation}\label{bianhao8}
B=\left\{e_{1}, \cdots, e_{k}\right\}
    \end{equation}
    is a maximal linearly independent set of $\left\{e_1, \cdots, e_n\right\}$ in the module $\mathbb{M}_{\mathcal{H}}.$
    Then for every $k+1\leq i \leq n,$ $\{e_1,\cdots,e_k,e_i\}$ is linearly dependent in the module $\mathbb{M}_{\mathcal{H}},$ hence there exist ploynomials $p_{1,i },\cdots, p_{k,i }$ and nonzero polynomial $p_{i,i }$
    such that \begin{equation}\label{baohanB}
\sum\limits_{j=1}^{k} p_{j,i } e_j+ p_{i,i } e_i=0.
    \end{equation}
    Write $$\mathfrak{C}=\{\text{deg}\, p_{j,i } : 1\leq j\leq k, k+1\leq i \leq n\}\cup \{\text{deg}\, p_{i,i }:  k+1\leq i \leq n\},$$
    and \begin{equation}\label{n_1}
n_1=\max \{r: r\in \mathfrak{C}\}.
    \end{equation}
  For $m>n_1$, $k+1\leq i \leq n,$ set $$
S_m=\left\{f \in \mathcal{P}_d: \operatorname{deg} f \leq m\right\} ~ \text{and}~ S_{m, i}=\left\{f \cdot p_{i,i } : f \in \mathcal{P}_d, \operatorname{deg} f \leq m-n_1\right\}.$$
It is evident that $S_{m, i}$ is a subspace of $S_m$. We can thus have a subspace, denoted by $S_{m, i}^{\prime}$, such that
\begin{equation}\label{S_m}
S_m = S_{m, i}^{\prime} \oplus S_{m, i}.
\end{equation}
For $m\geq 1,$ set
\begin{equation}\label{qd}
q_m(x)=\frac{(x+1)(x+2)\cdots(x+m)}{m!}.
\end{equation}
To prove \eqref{oula}, we need the following two lemmas.
\begin{lem}\label{Lemma 3.3}
For $k+1\leq i \leq n,$
$$ \lim\limits_{m\rightarrow \infty}\frac{\operatorname{dim}\left\{f \cdot e_i : f\in S_{m, i}^{\prime}\right\}}{q_d(m)} =0.$$
\end{lem}
\begin{proof}
For $m>n_1$, by \eqref{S_m},
 $$\operatorname{dim} S_{m, i}^{\prime}=\operatorname{dim} S_m-\operatorname{dim} S_{m, i}=q_d(m)-q_d\left(m-n_1\right).$$
Therefore
$$
\begin{array}{rlrl}
&\operatorname{dim}\left\{f \cdot e_i : f\in S_{m, i}^{\prime}\right\}
\leq \operatorname{dim} S_{m, i}^{\prime} = q_d(m)-q_d\left(m-n_1\right).
\end{array}
$$
 By the definition of $q_d(x)$ \eqref{qd}, there is a positive constant  $C_d$ such that
$$
\operatorname{dim}\left\{f \cdot e_i : f\in S_{m, i}^{\prime}\right\}  \leq q_d(m)-q_d\left(m-n_1\right)\leq C_d m^{d-1}.
$$
This implies that
$$\lim _{m \rightarrow \infty} \frac{\operatorname{dim}\operatorname{span}\left\{f \cdot e_i : f\in S_{m, i}^{\prime}\right\} }{q_d(m)} =0.$$
\end{proof}
 For any $\alpha = (\alpha_1, \alpha_2, \cdots, \alpha_d) \in \mathbb{N}^d $, set
 $$|\alpha|=\sum\limits_{i=1}^{d}\alpha_i, \quad \alpha!=\alpha_1! \cdots \alpha_d!,$$
 and $$z^\alpha = z_1^{\alpha_1} z_2^{\alpha_2} \cdots z_d^{\alpha_d}, \quad z \in \mathbb{B}_d.$$
\begin{lem}\label{Lemma 3.4}For $m\geq 1,$
$$\frac{\operatorname{dim} \operatorname{span}\left\{f \cdot \zeta : f \in \mathcal{P}_d, \operatorname{deg} f \leq m, \zeta \in B\right\}}{q_d(m)} =k.$$
\end{lem}
\begin{proof}
Set $$
F_m=\{z^{\alpha} : |\alpha| \leq m\},
$$
and
$$
B_m=\left\{f \cdot e : f \in F_m, e\in B\right\}.
$$
It suffices to show that $B_m$ is a linear basis of $\operatorname{span}\left\{f \cdot \zeta : f \in \mathcal{P}_d, \operatorname{deg} f \leq m, \zeta \in B\right\}.$
In fact, obviously
$$\operatorname{span} B_m=\operatorname{span}\left\{f \cdot \zeta : f \in \mathcal{P}_d, \operatorname{deg} f \leq m, \zeta \in B\right\}.$$
We only need to prove $B_m$ is linearly independent.
Write $$B_m=\{f_j \cdot e_i: 1\leq j\leq ^\sharp F_m, 1 \leq i \leq k\},$$
where $^\sharp F_m$ is the number of the elements in $F_m.$
Assume that for $\lambda_{i,j}\in \mathbb{C},$
$$\sum\limits_{i=1}^{k}\sum\limits_{j=1}^{^\sharp F_m}\lambda_{i,j} f_j e_i=0,$$ then from the fact that $B$
 is a maximal linearly independent set of $\left\{e_1, \cdots, e_n\right\}$ in the module $\mathbb{M}_{\mathcal{H}},$
 $\sum\limits_{j=1}^{^\sharp F_m}\lambda_{i,j} f_j =0$ for $1\leq i\leq k.$
  Notice that $F_m$ is linearly independent,
therefore $\lambda_{i,j}=0.$

It is easy to see that $^\sharp F_m= q_d(m)$ and $^\sharp B_m=k q_d(m).$
Then
$$
\operatorname{dim} \operatorname{span}\left\{f \cdot \zeta : f \in \mathcal{P}_d, \operatorname{deg} f \leq m, \zeta \in B\right\}= ^\sharp B_m=k q_d(m),
$$
which implies that
$$\frac{\operatorname{dim} \operatorname{span}\left\{f \cdot \zeta : f \in \mathcal{P}_d, \operatorname{deg} f \leq m, \zeta \in B\right\}}{q_d(m)} =k.$$
\end{proof}

Under the above preparations, we can prove \eqref{oula}.

\begin{prop}\label{fulu1}
Let \( \mathcal{H} \) be a \( d \)-contractive Hilbert module of finite rank, then
\begin{equation}
\chi(\mathcal{H}) = \lim_{m \rightarrow \infty} \frac{\operatorname{dim} \mathbb{M}_m}{q_d(m)}=d! \lim_{m \rightarrow \infty} \frac{\operatorname{dim} \mathbb{M}_m}{m^d},
\end{equation}
where
\[
\mathbb{M}_m = \operatorname{span}\{ f \cdot \zeta : f \in \mathcal{P}_d, \operatorname{deg} f \leq m, \zeta \in \operatorname{ran} \Delta_\mathcal{H} \}.
\]
\end{prop}
\begin{proof} By \eqref{oulashu},
$$
\chi(\mathcal{H})=\operatorname{rank}\mathbb{M}_{\mathcal{H}}.
$$
Notice that $\{e_i\}_{i=1}^{n}$ generates $\mathbb{M}_{\mathcal{H}},$ then by Lemma \ref{lemF},
$$
\chi(\mathcal{H})=\operatorname{rank}(\mathbb{M}_{\mathcal{H}})=k.
$$
Hence it suffices to prove that
$$
\lim _{m \rightarrow \infty} \frac{\operatorname{dim} \mathbb{M}_m}{q_d(m)}=k.
$$
Set $$V_i=\left\{f \cdot e_i : f \in \mathcal{P}_d, \operatorname{deg} f \leq m \right\},\quad k+1 \leq i \leq n.$$
For $k+1\leq i \leq n,$
by \eqref{baohanB} and \eqref{n_1}, we have
 $$\left\{f p_{i,i }\cdot e_i : f \in \mathcal{P}_d, \operatorname{deg} f \leq m-n_1\right\} \subseteq \operatorname{span} \left\{f \cdot \zeta : f \in \mathcal{P}_d, \operatorname{deg} f \leq m, \zeta \in B\right\},$$
then
$$
\begin{aligned}
V_i&
=\left\{f \cdot e_i : f \in S_ m \right\}\\&=
\left\{f \cdot e_i : f \in S_{m, i} \oplus S_{m, i}^{\prime} \right\}\\&=\left\{f \cdot e_i : f \in  S_{m, i}\right\}+\left\{f \cdot e_i : f \in S_{m, i}^{\prime}\right\}
\\&=\left\{f p_{i,i }\cdot e_i : f \in \mathcal{P}_d, \operatorname{deg} f \leq m-n_1\right\}+\left\{f \cdot e_i : f \in  S_{m, i}^{\prime}\right\}
\\& \subseteq \operatorname{span} \left\{f \cdot \zeta : f \in \mathcal{P}_d, \operatorname{deg} f \leq m, \zeta \in B\right\}+\left\{f \cdot e_i : f \in  S_{m, i}^{\prime}\right\}.
\end{aligned}
$$
Hence
$$
\begin{aligned}
 & \operatorname{span}\left\{f \cdot \zeta : f \in \mathcal{P}_d, \operatorname{deg} f \leq m, \zeta \in \operatorname{ran}\Delta_\mathcal{H}\right\}
 \\&=   \operatorname{span}\left\{f \cdot \zeta : f \in \mathcal{P}_d, \operatorname{deg} f \leq m, \zeta \in A\right\}
\\
& = \operatorname{span}\left\{f \cdot \zeta : f \in \mathcal{P}_d, \operatorname{deg} f \leq m, \zeta \in B\right\}+V_{k+1}+\cdots+V_n\\&
 \subseteq \operatorname{span} \left\{f \cdot \zeta : f \in \mathcal{P}_d, \operatorname{deg} f \leq m, \zeta \in B\right\}+\left\{f \cdot e_{k+1} : f \in  S_{m, k+1}^{\prime}\right\}+\cdots+\left\{f \cdot e_{n} : f \in  S_{m,n}^{\prime}\right\}.
\end{aligned}
$$
Therefore
$$
\begin{aligned}\label{rlc}
\operatorname{dim} \operatorname{span}&\left\{f \cdot \zeta : f \in \mathcal{P}_d, \operatorname{deg} f \leq m, \zeta \in B\right\}\\&\leq \operatorname{dim} \mathbb{M}_m  \\&=  \operatorname{dim} \operatorname{span}\left\{f \cdot \zeta : f \in \mathcal{P}_d, \operatorname{deg} f \leq m, \zeta \in \operatorname{ran}\Delta_\mathcal{H}\right\}
\\
& \leq  \operatorname{dim} \operatorname{span}\left\{f \cdot \zeta : f \in \mathcal{P}_d, \operatorname{deg} f \leq m, \zeta \in B\right\}+\sum\limits_{i=k+1}^n \operatorname{dim} \left\{f \cdot e_i : f \in  S_{m, i}^{\prime}\right\}.
\end{aligned}
$$
By Lemma \ref{Lemma 3.3} and Lemma \ref{Lemma 3.4}, we have
$$
\lim _{m \rightarrow \infty} \frac{\operatorname{dim} \mathbb{M}_m}{q_d(m)}=k.
$$
\end{proof}

Next, we will prove Proposition \ref{2.151}. The ideas and techniques come from \cite{Arveson curvature}.
 Let $\mathcal{H}$ be an $\omega$-contractive Hilbert module of finite rank.
For simplicity, write $R=\text{ran}\,\Delta_\mathcal{H}.$
Let $H^2(\partial \mathbb{B}_m)$ be the Hardy spaces on the unit ball, and $P_{H_m^2 \otimes R}$ be the projection from $H^2 \otimes R$ onto $H_m^2 \otimes R.$
Considering the linear operator $A_m: E \rightarrow H^2\left(\partial\mathbb{B}_m\right) \otimes R$ defined by
$$A_m \zeta=b_m P_{H_m^2 \otimes R} \Phi(1 \otimes \zeta), \quad \zeta \in E,$$
where $\Phi$ is defined in \eqref{fuluxuyao}
and $$b_m: H_m^2 \otimes R \rightarrow H^2\left(\partial\mathbb{B}_m\right) \otimes R$$ is the natural inclusion mapping.
To continue, we need some lemmas.
\begin{lem}\label{lemma74}
$$
\begin{aligned}
K_m(\mathcal{H})=\operatorname{rank} \mathcal{H}-\operatorname{trace}\left(A_m A_m^*\right).
\end{aligned}
$$
\end{lem}
\begin{proof}   Let $\hat{\Phi}=P_{H_m^2 \otimes R} \Phi P_{H_m^2 \otimes E}.$ Then for every $\zeta_1 \in E, \zeta_2 \in R, z^{(m)} \in \mathbb{B}_m$,
$$
\begin{array}{rlc}
\langle\hat{\Phi}(z^{(m)}) \zeta_1, \zeta_2\rangle & =  \langle\hat{\Phi}\left(1 \otimes \zeta_1\right), \mathcal{K}_{z^{(m)}} \otimes \zeta_2\rangle \\
 & =\langle P_{H_m^2 \otimes R} \Phi P_{H_m^2 \otimes E}\left(1 \otimes \zeta_1\right),  \mathcal{K}_{z^{(m)}} \otimes \zeta_2\rangle \\
 & =\langle\Phi(1 \otimes \zeta_1),  \mathcal{K}_{z^{(m)}} \otimes \zeta_2\rangle \\
& =\langle\Phi( z^{(m)}) \zeta_1, \zeta_2\rangle.
\end{array}
$$
So
$$
A_m \zeta(z^{(m)})=\hat{\Phi}(1 \otimes \zeta)(z^{(m)})=\hat{\Phi}(z^{(m)}) \zeta=\Phi( z^{(m)}) \zeta,
$$
and for almost $z^{(m)}\in \partial \mathbb{B}_m,$
$$
\begin{aligned}
\operatorname{trace}(\Phi( z^{(m)}) \Phi( z^{(m)})^*)=\operatorname{trace}(\Phi( z^{(m)})^* \Phi( z^{(m)}))=\sum\limits_{n=1}^{\infty}\|\Phi( z^{(m)}) \mathfrak{e}_n\|^2=\sum\limits_{n=1}^{\infty}\|A_m \mathfrak{e}_n( z_1,\cdots,z_m)\|^2,
\end{aligned}
$$
where $\{\mathfrak{e}_n\}_{i=1}^{\infty}$ is a basis of $E.$
Integrating over $\partial \mathbb{B}_m,$ we have
$$
\begin{aligned}
\int_{\partial \mathbb{B}_m} \operatorname{trace}(\Phi( z^{(m)}) \Phi( z^{(m)})^*)d \sigma_m&=\sum\limits_{n=1}^{\infty} \int_{\partial \mathbb{B}_m}\left\|A_m \mathfrak{e}_n(z^{(m)})\right\|^2 d \sigma_m\\&=\sum\limits_{n=1}^{\infty}\left\|A_m \mathfrak{e}_n\right\|_{H^2\left(\partial \mathbb{B}_m\right) \otimes R}^2
\\&=\operatorname{trace}\left(A_m^* A_m\right)\\&=\operatorname{trace}\left(A_m A_m^*\right).
\end{aligned}
$$
Then by \eqref{mudi2}, \eqref{mudi3} and \eqref{fulu4.9},
 $$
\begin{aligned}
K_{m}(\mathcal{H})=\int_{\partial\mathbb{ B}_{m}} \operatorname{trace}( I_{R}-\Phi( z^{(m)}) \Phi( z^{(m)})^{*}) d\sigma_m=\operatorname{rank} \mathcal{H}-\operatorname{trace}\left(A_m A_m^*\right).
\end{aligned}
$$
The proof is complete.
\end{proof}
Now, we define a linear map $\Gamma_m:\mathcal{B}(H_m^2 \otimes R)\rightarrow \mathcal{B}(H^2(\partial\mathbb{B}_m) \otimes R)$ as follows:
$$\Gamma_m(X)=b_m X b_m^*.$$
Let $Z_1,\cdots,Z_m$ be the canonical operators of the Hardy module $H^2(\partial\mathbb{B}_m) \otimes R.$
The linear map $d\Gamma_m:\mathcal{B}(H_m^2 \otimes R)\rightarrow \mathcal{B}(H^2(\partial\mathbb{B}_m) \otimes R)$ is defined
as follows:
$$d\Gamma_m(X)=\Gamma_m(X)-\sum\limits_{k=1}^{m}Z_k\Gamma_m(X)Z_k^*.$$
\begin{lem}\label{Ad}
$$
d \Gamma_m\left(P_{H_m^2 \otimes R} \Phi \Phi^* P_{H_m^2 \otimes R}\right)=A_m A_m^*
$$
\end{lem}
\begin{proof}
Notice that for $\xi \in H^2\left(\partial \mathbb{B}_m\right) \otimes R, \,\eta \in E,$
$$
\begin{aligned}
\langle1 \otimes A_m^* \xi, 1 \otimes \eta\rangle= & \langle A_m^* \xi, \eta\rangle \\
= & \langle\xi, A_m \eta\rangle \\
= & \langle\xi, b_m P_{H_m^2 \otimes R} \Phi P_{H_m^2 \otimes E}(1 \otimes \eta)\rangle \\
= &\langle P_{H_m^2 \otimes E} \Phi^* P_{H_m^2 \otimes R} b_m{ }^* \xi, 1 \otimes \eta\rangle.
\end{aligned}
$$
It follows that
$$
\begin{aligned}
\langle A_m A_m^* \xi, \eta\rangle= & \langle1 \otimes A_m^* \xi, 1 \otimes A_m^* \eta\rangle \\
= & \langle P_{H_m^2 \otimes E} \Phi^* P_{H_m^2 \otimes R} b_m^* \xi, 1 \otimes A_m^* \eta\rangle \\
= & \langle\Phi^* b_m^* \xi, 1 \otimes A_m^* \eta\rangle \\
= & \langle\Phi^* b_m^* \xi,\left(E_0 \otimes I_E\right)\left(1 \otimes A_m^* \eta\right)\rangle \\
= & \langle \left(E_0 \otimes I_E\right) \Phi^* b_m^* \xi, 1 \otimes A_m^* \eta\rangle \\
= & \langle \left(E_0 \otimes I_E\right) \Phi^* b_m^* \xi, P_{H_m^2 \otimes E} \Phi^* P_{H_m^2 \otimes R} b_m^* \eta\rangle \\
= & \langle b_m P_{H_m^2 \otimes R} \Phi P_{H_m^2 \otimes E}\left(E_0 \otimes I_E\right) \Phi^* b_m^* \xi, \eta\rangle \\
= &\langle b_m P_{H_m^2 \otimes R} \Phi\left(E_0 \otimes I_E\right) \Phi^* P_{H_m^2 \otimes R} b_m^* \xi, \eta\rangle.
\end{aligned}
$$
Thus, from the fact that $\Phi \in hom(H^{2} \otimes E, H^{2} \otimes R),$
we have
$$
\begin{aligned}
A_m A_m^* & =b_m P_{H_m^2 \otimes R} \Phi\left(E_0 \otimes I_E\right) \Phi^* P_{H_m^2 \otimes R} b_m^* \\
& =\left(b_m P_{H_m^2 \otimes R} \Phi\right)\left(I_{H^2 \otimes E}-\sum_{i=1}^{\infty} \mathfrak{T}_i \mathfrak{T}_i^*\right)\left(b_m P_{H_m^2 \otimes R} \Phi\right)^* \\
& =\left(b_m P_{H_m^2 \otimes R} \Phi\right)\left(b_m P_{H_m^2 \otimes R} \Phi\right)^*-\sum_{i=1}^m\left(b_m P_{H_m^2 \otimes R} \Phi\right) \mathfrak{T}_i \mathfrak{T}_i^*\left(b_m P_{H_m^2 \otimes R} \Phi\right)^* \\
& =\Gamma_m\left(P_{H_m^2 \otimes R} \Phi \Phi^* P_{H_m^2 \otimes R}\right)-\sum_{i=1}^m Z_i\left(b_m P_{H_m^2 \otimes R} \Phi\right)\left(b_m P_{H_m^2 \otimes R} \Phi\right)^* Z_i^* \\
& =\Gamma_m\left(P_{H_m^2 \otimes R} \Phi \Phi^* P_{H_m^2 \otimes R}\right)-\sum_{i=1}^m Z_i \Gamma_m\left(P_{H_m^2 \otimes R} \Phi \Phi^* P_{H_m^2 \otimes R}\right) Z_i^* \\
& =d \Gamma_m\left(P_{H_m^2 \otimes R} \Phi \Phi^* P_{H_m^2 \otimes R}\right),
\end{aligned}
$$
where $\mathfrak{T}_i=S_{z_i}\otimes I_{E}.$
\end{proof}
Set $\phi_m(X)=\sum\limits_{k=1}^m T_k X T_k^*.$\\
$Proof~of~Proposition~ \ref{2.151}.$ By Lemma \ref{Ad} and \cite[Theorem 3.10]{Arveson curvature}, we have
$$
\begin{aligned}
\operatorname{trace}(A_m A_m^*)&=\operatorname{trace}(d \Gamma_m(P_{H_m^2 \otimes R} \Phi \Phi^* P_{H_m^2 \otimes R}))\\&=\operatorname{rank \,\mathcal{H}} \lim _{n \rightarrow \infty} \frac{\operatorname{trace}\left(P_{H_m^2 \otimes R} \Phi \Phi^* P_{H_m^2 \otimes R} E_{n,m}\right)}{\operatorname{trace}\left(E_{n,m}\right)},
\end{aligned}
$$
where $E_{n,m} $ is the projection from $H_m^2 \otimes R$ onto its subspace of homogeneous (vector-valued) functions of degree $n$.
Since $L^* L+\Phi \Phi^*=I,$ by Lemma \ref{lemma74},
$$
\begin{aligned}
K_m(\mathcal{H}) & =\operatorname{rank} \mathcal{H}-\operatorname{trace}\left(A_m A_m^*\right) \\
&  =\operatorname{rank} \mathcal{H}\lim _{n \rightarrow \infty} \left( 1-\frac{\operatorname{trace}(P_{H_m^2 \otimes R} \Phi \Phi^* P_{H_m^2 \otimes R} E_{n,m})}{\operatorname{trace}\left(E_{n,m}\right)}\right)\\
& =\operatorname{rank} \mathcal{H} \lim _{n \rightarrow \infty} \frac{\operatorname{trace}\left(P_{H_m^2 \otimes R} L^* L P_{H_m^2 \otimes R} E_{n,m}\right)}{\operatorname{trace}\left(E_{n,m}\right)}\\&=\lim _{n \rightarrow \infty} \frac{\operatorname{trace}\left(P_{H_m^2 \otimes R} L^* L P_{H_m^2 \otimes R} E_{n,m}\right)}{q_{m-1}(n)}\\&=\lim _{n \rightarrow \infty} \frac{\operatorname{trace}\left(L P_{H_m^2 \otimes R} E_{n,m} P_{H_m^2 \otimes R} L^*\right)}{q_{m-1}(n)}\\&=\lim _{n \rightarrow \infty} \frac{\operatorname{trace}\left(\phi_m^n\left(\Delta_{\mathcal{H}}^2\right)\right)}{q_{m-1}(n)}.
\end{aligned}
$$
The last equation follows from the reasoning below.
For $\eta \in \mathcal{H}$, by Theorem \ref{2.6},
$$
L^* \eta=\left(\zeta_0, \zeta_1, \cdots\right),\quad \zeta_n=\sum_{i_1, i_2, \cdots i_n=1}^{\infty} e_{i_1} \otimes \cdots \otimes e_{i_n} \otimes \Delta_{\mathcal{H}} T_{i_n}^* \cdots T_{i_1}^* \eta.
$$
Let $\zeta_{m,n}=\sum\limits_{i_1, i_2, \cdots i_n=1}^m e_{i_1} \otimes \cdots \otimes e_{i_n} \otimes\Delta_{\mathcal{H}} T_{i_n}^* \cdots T_{i_1}^* \eta,$
then
$$
\begin{aligned}
L P_{H_m^2 \otimes R} E_{n,m} P_{H_m^2 \otimes R} L^* \eta &=L P_{H_m^2 \otimes R} E_{n,m} P_{H_m^2 \otimes R}\left(\zeta_0, \zeta_1, \cdots\right) \\&=L P_{H_m^2 \otimes R} E_{n,m}\left(\zeta_{m,0}, \zeta_{m,1}, \cdots\right) \\
& =L P_{H_m^2 \otimes R} \zeta_{m,n}  \\&=L \zeta_{m,n} \\&=\sum_{i_1, \cdots, i_n=1}^m T_{i_1} \cdots T_{i_n} \Delta_{\mathcal{H}}^2 T_{i_n}^* \cdots T_{i_1}^* \eta\\&=\phi_m^n\left(\Delta_{\mathcal{H}}^2\right) \eta.
\end{aligned}
$$
Therefore
$$
K_m(\mathcal{H})=\lim _{n \rightarrow \infty} \frac{\operatorname{trace}\left(\phi_m^n\left(\Delta_{\mathcal{H}}^2\right)\right)}{q_{m-1}(n)}=\lim _{n \rightarrow \infty} \frac{\sum\limits_{k=0}^n \operatorname{trace}\left(\phi_m^k\left(\Delta_{\mathcal{H}}^2\right)\right)}{\sum\limits_{k=0}^n q_{m-1}(k)}
=\lim _{n \rightarrow \infty} \frac{\operatorname{trace}\left(\sum\limits_{k=0}^n \phi_m^k\left(\Delta_{\mathcal{H}}^2\right)\right)}{q_m(n)}.
$$

Finally, we will prove Proposition \ref{3.21}.\\
{\color{red}$Proof~of~Proposition ~\ref{3.21}.$}
Obviously,
$$
\begin{aligned}
 \mathbb{M}_{\mathcal{H}}^{m,n}&=span \left\{f \cdot \Delta_{\mathcal{H}} \zeta \mid f \in \mathcal{P}_m, \operatorname{deg} f \leq n, \zeta \in H\right\} \\
& =span \left\{T_{i_1} \cdots T_{i_k} \Delta_{\mathcal{H}} \zeta \mid \zeta \in H, 1 \leq i_1, \cdots, i_k \leq m, k=0,1, \cdots, n\right\} 
\end{aligned}
$$
is a linear submanifold of $\mathcal{H}.$
Then
$$
\begin{aligned}
 {\mathbb{M}_{\mathcal{H}}^{m,n}}^{\perp}&=\left\{\eta \in \mathcal{H}\left|\sum_{k=0}^n \sum_{i_1=1}^{m}\cdots\sum_{i_k=1}^m|\left\langle T_{i_1} \cdots T_{i_k} \Delta_{\mathcal{H}} \zeta, \eta \right \rangle\mid=0, \forall \zeta \in \mathcal{H}\right\}\right. \\
& =\left\{\eta \in \mathcal{H} \left|\sum_{k=0}^n \sum_{i_1=1}^{m}\cdots\sum_{i_k=1}^m\|\Delta_{\mathcal{H}} T_{i_k}^* \cdots T_{i_1}^* \eta\|^2=0\right\}\right.\\&=\left\{\eta \in \mathcal{H}\left| \sum_{k=0}^n\left\langle\phi_m^k(\Delta_{\mathcal{H}}^2) \eta, \eta\right\rangle=0\right\} \right.\\
& = \operatorname{ker}\left(\sum_{k=0}^n \phi_m^k(\Delta_{\mathcal{H}}^2)\right). 
\end{aligned}
$$
Hence by \eqref{zuihou1}, $$\chi_m(\mathcal{H})=m ! \lim _{n \rightarrow \infty} \frac{\operatorname{dim} \mathbb{M}_{\mathcal{H}}^{m,n}}{n^m}=m ! \lim _{n \rightarrow \infty} \frac{\operatorname{dim} \operatorname{ran}\left(\sum\limits_{k=0}^n \phi_m^k\left(\Delta_{\mathcal{H}}^2\right)\right)}{n^m}.$$

\textbf{Acknowledgment.} The authors thank the referees for helpful suggestions, which make this paper more readable.

 \noindent{Penghui Wang, School of Mathematics, Shandong University, Jinan 250100, Shandong, P. R. China, Email:
phwang@sdu.edu.cn}

\noindent{Ruoyu Zhang, School of Mathematics, Shandong University, Jinan 250100, Shandong, P. R. China, Email:
202311805@mail.sdu.edu.cn
}

 \noindent{Zeyou Zhu, School of Mathematics, Shandong University, Jinan 250100, Shandong, P. R. China, Email:
201911795@mail.sdu.edu.cn}

\end{document}